\documentclass[a4paper,twoside,12pt,leqno,
]{article}
\usepackage{amsmath,amsthm,amssymb,enumerate}
\usepackage{graphicx,  float, afterpage,array,multicol,eucal, url,color}

\swapnumbers
\theoremstyle{plain}
\newtheorem{theorem}{Theorem}[section]
\newtheorem{lemma}[theorem]{Lemma}

\newtheorem{AAP}[theorem]{Theorem\normalfont~\cite[Theorem~4.14]{AAP}}
\newtheorem{AAP2}[theorem]{The Abrams-\'Anh-Pardo
Theorem\normalfont~\cite{Pardo}}

\theoremstyle{definition}

\newtheorem{definitions}[theorem]{Definitions}
\newtheorem{example}[theorem]{Example}

\newtheorem{notation}[theorem]{Notation}

\newtheorem{remarks}[theorem]{Remarks}
\newtheorem{conclusions}[theorem]{Conclusions}

\newtheorem{background}[theorem]{Background}

\newcommand{\gp}[2]{\gen{{#1}\mid #2}}
\newcommand{\abs}[1]{\left\lvert#1\right\rvert} 
\newcommand{\gen}[1]{\langle\mkern3mu#1\mkern3mu\rangle}

\newcommand{\Mat}[1]{\operatorname{M}_{\mkern1mu#1\mkern-2mu}} 
\newcommand{\PU}[1]{\operatorname{PU}_{\mkern1mu#1\mkern-2mu}}

\DeclareMathOperator{\transp}{transp}

\DeclareMathOperator{\Fix}{Fix}
\DeclareMathOperator{\Germs}{Germs}
\DeclareMathOperator{\rep}{rep}
\def \L {\mathbb{L}}

\def \integers {\mathbb{Z}}

\def\iso{\cong}
\def\d1{\discretionary{-}{}{-}}
\def\coloneq{\mathrel{\mathop\mathchar"303A}\mkern-1.2mu=}

\newcommand\mapsfrom{\mathrel{\reflectbox{\ensuremath{\mapsto}}}}

\renewcommand{\phi}{\varphi}
\renewcommand{\epsilon}{\varepsilon}
\renewcommand{\le}{\leqslant}
\renewcommand{\ge}{\geqslant}
\renewcommand{\mod}{\text{mod }}

\begin{document}

\pagestyle{myheadings}
\markboth{Isomorphisms of  Brin-Higman-Thompson groups}
{Warren Dicks and Conchita Mart\'{\i}nez-P\'erez}

\title{Isomorphisms of  Brin-Higman-Thompson groups}

\author{
Warren Dicks\footnote{Partially supported
    by Spain's Ministerio de Ciencia e Innovaci\'on
through Project  MTM2008-01550.},\,\,
  Conchita Mart\'{\i}nez-P\'erez\footnote{Partially
supported by the Gobierno de Arag\'on, the
European Regional Development Fund and
Spain's Ministerio de Ciencia e Innovaci\'on
through Project  MTM2010-19938-C03-03.}}

\date{\small\today}

\maketitle

\noindent\textbf{Abstract.} Let  $m,\,m',\,r,\, r',t,\,t'$ be
positive integers with $r,\,r' \ge 2$.
Let $\L_r$ denote the ring that is universal with an
invertible $1{\times}r$ matrix.    Let $\Mat{m}(\,\L_r^{\otimes t})$ denote the
ring of $m \times m$ matrices over the tensor product of $t$ copies of $\L_r$.
In a natural way, $\Mat{m}(\,\L_r^{\otimes t})$
is a partially ordered ring with involution.  Let
$\PU{m}(\,\L_r^{\otimes t})$ denote the
group of positive unitary elements.  We  show that
$\PU{m}(\mkern2mu\L_r^{\otimes t})$ is isomorphic to the
Brin-Higman-Thompson group~$t V_{r,m}$;
the case $t{\,=\,}1$ was found by Pardo, that is,
$\PU{m}(\,\L_r)$ is isomorphic to the  Higman-Thompson group~$V_{r,m}$.

We survey arguments of Abrams, \'Anh, Bleak, Brin, Higman, Lanoue,
Pardo, and  Thompson
that prove that $t' V_{r',m'}\iso t V_{r,m} $ if and only if
$r'{\,=\,}r$, $t'{\,=\,}t$ and $ \gcd(m',r'{-}1) = \gcd(m,r{-}1)$
\,\,(if and only if
$\Mat{m'}(\,\L_{r'}^{\otimes t'})$ and $\Mat{m}(\L_r^{\otimes t})$
are isomorphic as partially ordered rings
with involution).

\medskip

{\footnotesize
\noindent \emph{2010 Mathematics Subject Classification.} Primary:  16S10;
Secondary: 20E32, 20B27, 16S50, 20B22.

\noindent \emph{Key words.}  Brin-Higman-Thompson group. Isomorphism.
Leavitt ring.
Positive unitary group.}

\bigskip

\section{Introduction}\label{sec:1}

The notation we use will be explained in the next section.

Throughout, fix $r,\,r' \in [2{\uparrow}\infty[\,$, $m, \, m', \, t,\,t' \in
[1{\uparrow}\infty[\,$, and fix symbols  $x$ and $y$, and let
\begin{equation*}
\L_r \coloneq \integers \gp{x_{[1{\uparrow}r]},y_{[1{\uparrow}r]}} {
x_{[\mkern-3mu[1{\uparrow}r]\mkern-3mu]}^{\transp}\cdot
    y_{[\mkern-3mu[1{\uparrow}r]\mkern-3mu]}   = \mathbf{I}_r \text{ and
}\,  y_{[\mkern-3mu[1{\uparrow}r]\mkern-3mu]}
\cdot x_{[\mkern-3mu[1{\uparrow}r]\mkern-3mu]}^{\transp} =
    1}.
\end{equation*}
Thus, for example, $\L_2 = \integers \gp{x_{1},x_{2},y_1,y_2} {
\scriptstyle\begin{pmatrix}
x_1y_1 &x_1y_2\\
x_2y_1 & x_2y_2
\end{pmatrix}=\begin{pmatrix}
1 &0\\
0 & 1
\end{pmatrix}, \textstyle y_1x_1 + y_2x_2 = 1}.$
We use the symbol $\L$  in recognition of  Leavitt's pioneer work on
these rings in~\cite{Leavitt0},~\cite{Leavitt2}.
We let
$\L_r^{\otimes t} \coloneq \L_r \otimes_\integers \L_r \otimes_\integers
\cdots\otimes_\integers \L_r$,
the ring obtained by forming the tensor product over $\integers$ of
$t$ copies of $\L_r$. We shall be interested in the $m{\times}m $ matrix ring
    $\Mat{m}(\,\L_r^{\otimes t})$.
The mnemonic is that $r$ is for `ring',  $t$ is for `tensor', and
$m$ is for `matrix'.
In a natural way, $\Mat{m}(\,\L_r^{\otimes t})$ is
    a partially ordered ring with involution.  We then let
$\PU{m}(\,\L_r^{\otimes t})$
denote the subgroup of positive unitary elements in the group of units of
$\Mat{m}(\,\L_r^{\otimes t})$.

Abrams, \'Anh and Pardo~\cite{AAP},~\cite{Pardo} found that if
    $\gcd(m',r{-}1) = \gcd(m,r{-}1)$, then
$\Mat{m'}(\,\L_r)$ and $\Mat{m}(\,\L_r)$ are isomorphic  as partially
ordered rings with
involution;
we shall observe that it then
follows easily that $\Mat{m'}(\,\L_r^{\otimes t})$ and
$\Mat{m}(\,\L_r^{\otimes t})$
are isomorphic  as partially ordered rings with involution, and that the groups
$\PU{m'}(\,\L_r^{\otimes t}) $ and
$\PU{m}(\,\L_r^{\otimes t})$
are isomorphic.
We shall give self-contained proofs  of all these isomorphisms.

Pardo~\cite{Pardo} discovered a connection between these rings and
certain  famous groups.
In~\cite{Higman}, Higman constructed a group $V_{r,m}$ with the properties that
   the abelianization of $V_{r,m}$  has order
$\gcd(2,r{-}1)$ and the
    derived group  of $V_{r,m}$  is a  finitely presentable,
infinite, simple group; see also~\cite{Scott}.  The group  $V_{2,1}$
is  Thompson's  group  $V.$
In \cite{Brown}, Brown showed that $V_{r,m}$ is of
type~$\operatorname{FP}_\infty$.
Pardo~\cite{Pardo} found that
$V_{r,m} \iso \PU{m}(\,\L_r)$; hence, if $\gcd(m',r{-}1) =
\gcd(m,r{-}1)$, then the above
    isomorphism   $\PU{m'}(\,\L_r ) \iso \PU{m}(\,\L_r )$
gives the converse of
    Higman's result that
$V_{r',m'}  \iso V_{r,m} $ only if $r' = r$ and $\gcd(m', r'{-}1) =
\gcd(m, r{-}1)$;
see~\cite[Theorem~6.4]{Higman}.

In~\cite[Section~4.2]{Brin}, Brin constructed a group
$tV_{r,m}$ which can be considered as a
$t$-dimensional analogue
of the Higman-Thompson group $V_{r,m} \,(=1V_{r,m})$.  In~\cite{Brin},
he proved that $2V_{2,1}$ is simple
   and  that $2V_{2,1} \not \iso V_{r,m}$ and other results.
In~\cite{Brin2}, he proved that $t V_{2,1}$ is simple.
In~\cite{HM}, Hennig and Matucci gave a finite presentation  of $t V_{2,1}$.
In~\cite{BL}, Bleak and Lanoue showed that $t'V_{2,1} \iso tV_{2,1}$
if and only if $t' = t$.
In~\cite{KochMPNuc}, a description of
$tV_{r,m}$ along the lines of Higman's
construction~\cite{Higman},~\cite{Scott},
was given, and it was used
to show that $2V_{2,1}$ and $3V_{2,1}$ are of type $\operatorname{FP}_\infty$.

The main purpose of this article is to show that
$tV_{r,m}\iso \PU{m}(\,\L_r^{\otimes t})$.
   Straightforward adaptations of known arguments
then show that the following are equivalent.\vspace{-2mm}
\begin{enumerate}[(a)]
\setlength\itemsep{-.2pt}
\item $r'=r$, $t'=t$ and $ \gcd(m',r'{-}1)=\gcd(m,r{-}1) $.
\item
$\Mat{m'}(\,\L_{r'}^{\otimes t'})$ and $\Mat{m}(\L_r^{\otimes t})$
are isomorphic as partially ordered rings
with involution.
\item  $ t' V_{r',m'} \iso t V_{r,m}$.
\end{enumerate}\vspace{-2mm}
Thus, with $r$ and $t$ fixed and $m$ varying,
the set of isomorphism classes of the
groups $t V_{r,m}$  is in bijective
correspondence with the set of  positive divisors of  $r{-}1$.

It may be of interest to ring theorists that there is a fourth equivalent statement:

\smallskip

\noindent $(\text{b}')$ $\Mat{m'}(\,\L_{r'}^{\otimes t'})$ and $\Mat{m}(\L_r^{\otimes t})$
are isomorphic as  rings.

\smallskip

\noindent Thus, with $r$ and $t$ fixed and $m$ varying,
the set of isomorphism classes of the
rings $\Mat{m}(\L_r^{\otimes t})$ is in bijective
correspondence with the set of  positive divisors of  $r{-}1$.

The structure of the article is as follows.
In the first part, we work exclusively with $\Mat{m}(\,\L_r^{\otimes t})$.

In Section~\ref{sec:2}, we summarize the notation that we shall be
using, and endow
$\Mat{m}(\,\L_r^{\otimes t})$ with the structure of a partially
ordered ring with
involution.

In Section~\ref{sec:5}, we give a  streamlined proof of  the crucial
Abrams-\'Anh-Pardo
result~\cite{AAP}  that if
$m > r \ge 3$ and $\gcd(m,r{-}1)=1$, then $\Mat{m}(\,\L_r)$ and $ \L_r$
are isomorphic
as partially ordered rings with involution.

In Section~\ref{sec:7},  following Pardo~\cite{Pardo}, we show that if
$\gcd(m',r{-}1) = \gcd(m,r{-}1)$ then
$\Mat{m'}(\,\L_r^{\otimes t}) $ and $ \Mat{m}(\,\L_r^{\otimes t})$
are isomorphic as partially ordered rings with involution, and, hence,
$\PU{m'}(\,\L_r^{\otimes t}) \iso \PU{m}(\,\L_r^{\otimes t})$.

In the second part of the article, we concentrate on $tV_{r,m}$.

In Section~\ref{sec:3},  we  prove our main result that
the Brin-Higman-Thompson group
$tV_{r,m}$ is isomorphic to $\PU{m}(\,\L_r^{\otimes t})$;
the case $t{\,=\,}1$ was found by Pardo~\cite{Pardo}, that is,
the  Higman-Thompson group~$V_{r,m}$ is isomorphic to
$\PU{m}(\,\L_r)$.
It then follows that if $\gcd(m',r{-}1) = \gcd(m,r{-}1)$, then
$tV_{r,m'} \iso tV_{r,m}$.

In Section~\ref{sec:hig}, we find that arguments of Higman    show that
if  $ t' V_{r',m'}\iso t V_{r,m}$, then
$r'=r$ and $\gcd(m',r'{-}1)=\gcd(m,r{-}1)$.

In Section~\ref{sec:BBL}, we find that arguments of Bleak, Brin,
Lanoue and Rubin
   show that if \mbox{$t' V_{r',m'}\iso t V_{r,m} $}, then
$t'{\,=\,}t$.

In Section~\ref{sec:summ}, we  summarize much
   of the foregoing by recording  the above   equivalence \mbox{(a) $
\Leftrightarrow$ (b) $ \Leftrightarrow$ (c)}.
We conclude with a sketch of unpublished results of Ara, Bell and Bergman
that show that $t$, $r$ and
$\gcd(m,r{-}1)$ are invariants of the isomorphism class of
$\Mat{m}(\,\L_{r}^{\otimes t})$
\textit{as ring}, and thus the foregoing equivalent conditions are
further equivalent to (b$'$).

\section{Notation}\label{sec:2}

We will find it useful to  have a vocabulary for intervals in $\integers$.

\begin{notation} Let $i$, $j \in \integers$.
We define the vector $$[\mkern-3mu[i{\uparrow}j]\mkern-3mu]\coloneq
    \begin{cases}
(i,i+1,\ldots, j-1, j) \in \integers^{j-i+1}   &\text{if $i \le j$,}\\
() \in \integers^0 &\text{if $i > j$.}
\end{cases}
$$
The underlying subset of $\integers$ will be denoted $[i{\uparrow}j]$.
Similar notation applies for $[i{\uparrow}\infty[$\,.

Let $v_k$ be an integer-indexed symbol.  We define the vector
$$v_{[\mkern-3mu[i{\uparrow}j]\mkern-3mu]}\coloneq  \begin{cases}
(v_i,v_{i+1},   \cdots, v_{j-1},   v_j)  &\text{if $i \le j$,}\\
()  &\text{if $i > j$.}
\end{cases}$$
The underlying set will be denoted  $v_{[ i{\uparrow}j] }$.

When $v_{k,k'}$ is a doubly indexed symbol, we write $$v_{[
i{\uparrow}j] \times [
i'{\uparrow}j']} \coloneq
\{v_{k,k'} \mid k \in [i{\uparrow}j], k' \in [i'{\uparrow}j']\}.$$

Let $R$ be a ring with a unit.

For any
subset $Z$ of $R$,
we let $ \langle \langle Z \rangle \rangle$ denote the multiplicative
submonoid of $R$ generated by~$Z$.

Suppose that $m$, $n \in [1{\uparrow}\infty[\,$.

We let $\null^{m}\mkern-3muR^{\mkern2mu n}$ denote the set of $m{\times}n$
matrices over $R$ and we write $\Mat{m}(R) \coloneq
\null^{m}\mkern-3muR^{\mkern2mu m}$.

For $i \in [1{\uparrow}m]$ and $j \in [1{\uparrow}n]$,
we let $e_{i,j} \in \null^{m}\integers^{n}$ denote the
$m {\,\times\,} n$ matrix whose $(i,j)$ coordinate is $1$, and all other
coordinates are zero.
This notation applies  only where the ranges of $i$ and $j$ are
clearly specified.
We think of $\null^{m}\mkern-3muR^{\mkern2mu n}$  as
$\null^{m}\integers^{n}\otimes_\integers R$ and
   use the same symbol $e_{i,j}$ to denote the image
in~$\null^{m}\mkern-3muR^{\mkern2mu n}$.

We define  an additive \textit{transpose  map} $\null^{m}\integers^{n} \to
\null^{n}\integers^{m}$, $U \mapsto U^\ast$,
such that $e_{i,j}^\ast \coloneq e_{j,i}$.
We endow $\null^{m}\integers^{n}$ with the structure of a partially
ordered abelian group
in which  the \textit{positive cone}
$\operatorname{P}(\,\null^{m}\integers^{n})$ (the
set of  elements $\ge 0$)  is the additive monoid generated by
$e_{[1{\uparrow}m] \times [1{\uparrow}n]}$.

In particular, $\Mat{m}(\integers) $ has the structure of a ring with
involution $p \mapsto p^\ast$, and the structure of
a partially ordered abelian group.  We note that
the positive cone \mbox{$\operatorname{P}_{m}( \integers ) \coloneq
\operatorname{P}(\,\null^{m}\integers^{m})$}
contains $1$ and is closed under
multiplication and the involution.  Thus  $\Mat{m}(\integers)$ has
the structure of a partially ordered
ring with involution.

\end{notation}

\begin{notation}
Throughout, fix $r \in [2{\uparrow}\infty[\,$,  and fix symbols  $x$
and $y$, and let
\begin{equation*}
\L_r \coloneq \integers \gp{x_{[1{\uparrow}r]},y_{[1{\uparrow}r]}} {
x_{[\mkern-3mu[1{\uparrow}r]\mkern-3mu]}^{\transp}\cdot
    y_{[\mkern-3mu[1{\uparrow}r]\mkern-3mu]}   = \mathbf{I}_r \text{ and
}\,  y_{[\mkern-3mu[1{\uparrow}r]\mkern-3mu]}
\cdot x_{[\mkern-3mu[1{\uparrow}r]\mkern-3mu]}^{\transp} = 1}.
\end{equation*}
Here $x_{[\mkern-3mu[1{\uparrow}r]\mkern-3mu]}$ and
$y_{[\mkern-3mu[1{\uparrow}r]\mkern-3mu]}$ are $1\,{\times}\,r$ row vectors,
    $x_{[\mkern-3mu[1{\uparrow}r]\mkern-3mu]}^{\transp}$ denotes
    the $r \,{\times}\,1$ transpose of
$x_{[\mkern-3mu[1{\uparrow}r]\mkern-3mu]}$,
and $\mathbf{I}_r$~denotes the $r \,{\times}\, r$ identity matrix.

Leavitt~\cite{Leavitt0} showed that each element of $\L_r$ has a
unique \textit{normal form},
which is an expression as a $\integers$-linear combination of
   elements of  $ \langle \langle  x_{[1{\uparrow}r]} \cup
y_{[1{\uparrow}r]}
\rangle \rangle$
which do not contain any contiguous subword of the form $x_sy_{s'}\, (=
\delta_{s,s'})$, $s,\, s' \in [1{\uparrow}r]$, or
$y_rx_r \,(= 1-\sum_{s\in[1{\uparrow}(r{-}1)]} y_sx_s)$.
By Leavitt's normal-form result, the multiplicative
monoid $ \langle \langle x_{[1{\uparrow}r]}
\rangle \rangle$ is freely generated by $x_{[1{\uparrow}r]}$,
and similarly for $ \langle \langle y_{[1{\uparrow}r]}  \rangle
\rangle$.

We endow $\L_r$  with the involution $p
\mapsto p^*$ which is
the unique anti-automorphism  which interchanges
$x_{[\mkern-3mu[1{\uparrow}r]\mkern-3mu]}$
and $y_{[\mkern-3mu[1{\uparrow}r]\mkern-3mu]}$.

We endow $\L_r$ with the structure of a partially ordered abelian group
in which the positive cone   $\operatorname{P}(\L_r)$
is the additive monoid generated by the set of monomials
$ \langle \langle  x_{[1{\uparrow}r]} \cup  y_{[1{\uparrow}r]}
\rangle \rangle$.
This is a partial order since a nonempty sum of monomials is not zero.
To see this notice that for any row vector of zeros and monomials,
some of which have positive $x$-degree,  multiplying on the right by
a suitable $y_i$  leaves a nonzero vector of smaller largest
$x$-degree, and the result follows
by induction.
We note that the positive cone contains $1$ and
is closed under multiplication and the involution.  Thus $\L_r$ has
been endowed with the
structure of a partially ordered ring with involution.

Let $m,\,n,\,t \in[1{\uparrow}\infty[\,$.

We extend the involutions on each of the $t{+}1$ factors  to the
\textit{conjugate-transpose map}
\mbox{$\null^m\integers ^{\mkern1mu n} \otimes_\integers \L_r^{\otimes t}
\to\null^n\integers ^{\mkern1mu m} \otimes_\integers \L_r^{\otimes t}
$}, $U \mapsto U^*$.
Recall that we identify $\null^m\integers ^{\mkern1mu n}
\otimes_\integers \L_r^{\otimes t}
= \null^m(\L_r^{\otimes t})^{\mkern1mu n} $.
Let $\operatorname{U}(\,\null^m(\L_r^{\otimes t})^{\mkern1mu n})$
denote the set
of $Y \in \null^m(\L_r^{\otimes t})^{\mkern1mu n}$ such that
$Y\cdot Y^\ast = \mathbf{I}_m$ and $Y^\ast \cdot Y  = \mathbf{I}_n$.
The elements of $\operatorname{U}(\,\null^m(\L_r^{\otimes
t})^{\mkern1mu n})$ are
called the \textit{unitary} $m\times n$ matrices over $\L_r^{\otimes t}$.
We write $\operatorname{U}_m(\, \L_r^{\otimes t}) \coloneq
\operatorname{U}(\,\null^m(\L_r^{\otimes t})^{\mkern1mu m})$, a subgroup of
the group of units of $\Mat{m}(\,\L_r^{\otimes t})$.

We extend the partial order on each of the $t{+}1$ factors to  all of
$\null^m\integers ^{\mkern1mu n} \otimes_\integers \L_r^{\otimes t} $
by taking as the
positive cone
$\operatorname{P}(\null^m\integers ^{\mkern1mu n} \otimes_\integers
\L_r^{\otimes t})$
the additive submonoid generated by the     product of the positive
cones of the
factors;  as before, a nonempty sum of
monomials is not zero.
    We write
    $\operatorname{P}_m( \,\L_r^{\otimes t} ) \coloneq
\operatorname{P}(\,\null^m(\L_r^{\otimes t})^{\mkern1mu m})$,
a multiplicative submonoid with involution
in the ring with involution  $\Mat{m}(\,\L_r^{\otimes t})$.
Thus  $\Mat{m}(\,\L_r^{\otimes t})$ has been endowed with the structure of a
partially ordered ring with involution.

Let $\PU{}(\,\null^m(\L_r^{\otimes t})^{\mkern1mu n}) \coloneq
\operatorname{P}(\,\null^m(\L_r^{\otimes t})^{\mkern1mu n})
\cap \operatorname{U}(\,\null^m(\L_r^{\otimes t})^{\mkern1mu n})$ and
$\PU{m}(  \L_r^{\otimes t}  ) \coloneq
\PU{}(\,\null^m(\L_r^{\otimes t})^{\mkern1mu m})$.
Then $\PU{m}(\,\L_r^{\otimes t})=
\operatorname{P}_m( \L_r^{\otimes t} ) \cap \operatorname{U}_m(
\L_r^{\otimes t} )$,
an intersection
of  multiplicative monoids with involution, and hence   itself a
multiplicative monoid with involution.
Since $\PU{m}( \L_r^{\otimes t} )$ lies in
$\operatorname{U}_{m}(\,\L_r^{\otimes t})$ the involution acts as inversion and
$\PU{m}(\,\L_r^{\otimes t}) $ is a multiplicative group. We call
$\PU{m}(\,\L_r^{\otimes t}) $ the \textit{group of
positive unitary  $m\times m$ matrices over~$\L_r^{\otimes t}$}.
\end{notation}

\section{The crucial ring isomorphism}\label{sec:5}

This following beautiful result of Abrams, \'Anh and Pardo has the
unusual property that it
shows that two naturally defined rings are isomorphic
without giving a natural reason, and there may not be one.
We shall be giving their proof but shall incorporate
   a permutation of $\integers$ that will  automate
much of their book-keeping.
Although the proof we shall give uses
  $r \ne 2$ and $r < m$,  we shall see in the next section
that the result holds without these restrictions.

\begin{AAP}\label{thm:aap}  Let  $r \in [3{\uparrow} \infty[\,$  and
\mbox{$m \in [(r{+}1){\uparrow}\infty[\,$} with
$\gcd(m,r{-}1)=1$. Then $\L_r$ and $\Mat{m}(\,\L_r)$  are isomorphic
as   partially ordered rings
with involution.
\end{AAP}

\begin{proof}  Let $\L \coloneq \L_r$. Define $\pi \colon \integers
\to \integers$ by
    $$i \mapsto \  i^\pi\coloneq \begin{cases}
i{+} r &\text{if } i  \equiv  0\, (\mod m),\\
i{+} r{-}2 &\text{if }i  \equiv  1 \,(\mod m),\\
i{+} r{-}1 &\text{if } i \not\equiv 0,1\,(\mod m).
\end{cases}$$
Thus $\pi$ shifts every element of $\integers$ up by $r{-}1$,
except that certain adjacent pairs $(\ell m, \ell m {+}1)$ are carried to
    $(\ell m{+} r, \ell m {+}r{-}1)$, that is, they are shifted by $r{-}1$
and then interchanged.
Notice that $\pi$ is bijective.

We claim that    $[2{\uparrow}r]$ is a set of $\gen{\pi}$-orbit
representatives in $\integers$.
Because $\pi$ shifts every element of $\integers$ up by at most $r$
and by at least $r{-}2 \,(\,{\ge}1)$,
it follows that each $\gen{\pi}$-orbit meets~$[1{\uparrow}r]$.  Now
\mbox{$1^\pi = r{-}1 \,(\,{\ge} 2)$} which lies in
   \mbox{$ [2{\uparrow}r].$}
Hence, each $\gen{\pi}$-orbit   meets $[2{\uparrow}r]$.
Since $2^\pi = r{+}1$ and $\pi$ shifts every element of $\integers$
up by at least $r{-}2$, we see that
no $\gen{\pi}$-orbit  meets $[2{\uparrow}r]$ twice.  This proves the claim.

Now any sequence of $r{-}1$ consecutive integers is a set of
$\gen{\pi}$-orbit representatives
unless two elements are in the same $\gen{\pi}$-orbit, and the only
pairs in the
same $\gen{\pi}$-orbit
that are at distance less than $r{-}1$ are of the form $\ell m{+}1
\mapsto \ell m{+\,}r{-}1$.
If our sequence of $r{-}1$ consecutive integers does not start at an
$\ell m{+}1$,
    then it cannot contain two elements in the same $\gen{\pi}$-orbit.
    Hence,  for each $k \in \integers$,
    $[(k{+}1){\uparrow}(k{+}r{-}1)]$ is a set of $\gen{\pi}$-orbit
representatives in~$\integers$
if and only if  $k \not\equiv 0\, \mod m $.

Let \mbox{$s \in [2{\uparrow}r]$} and $j \in [1{\uparrow}(m{-}1)]$.
    Since $\gcd(m,r{-}1)=1$,  we have \mbox{$(r{-}1)j \not\equiv 0\, \mod m $.}
Hence, $[(1{+}(r{-}1)j){\uparrow}((r{-}1)(j{+}1))]$  is a set of
$\gen{\pi}$-orbit representatives in $\integers$ and therefore contains
   a unique element in the $\gen{\pi}$-orbit of $s$.
We denote that element by
$s \# j$.  Thus $s \# j\,\, \in
\,\,[(1{+}(r{-}1)j){\uparrow}((r{-}1)(j{+}1))]$ and $(s \#
j)^{\gen{\pi}} = s^{\gen{\pi}}$.
In $\L$, define  \mbox{$y_{(s\#j) + m} \coloneq  y_1^{j-1}y_s$.}
Define \mbox{$y_{r+m-1} \coloneq y_1^{m-1}$.}
For each $k \in  [(r{+}m{-}1){\uparrow}(mr)]$,   define $x_k \coloneq
y_k^\ast$.
We claim that we have defined
$y_{[\mkern-3mu[(r+m-1){\uparrow}(mr)]\mkern-3mu]}$  with underlying
set
$ y_1^{[0{\uparrow}(m{-}2)]} y_{[2{\uparrow}r]} \cup \{ y_1^{m-1}\}$.
For each \mbox{$j \in [1{\uparrow}(m{-}1)]$}, varying $s\in [2{\uparrow}r]$,
we see that $$[2{\uparrow}r]\#j = [(1{+}(r{-}1)j){\uparrow}(r{-}1)(j{+}1)],$$
and then $$[2{\uparrow}r]\#j  + m =
[(1{+}(r{-}1)j{+}m){\uparrow}((r{-}1)(j{+}1){+}m)],$$
and we have defined
$y_{[\mkern-3mu[(1{+}(r{-}1)j + m){\uparrow}((r{-}1)(j{+}1)+m)]\mkern-3mu]}$
with underlying  set
$  y_1^{j-1}y_{[2{\uparrow}r]}$.
By then varying~$j \in [1{\uparrow}(m{-}1)]$, we
obtain    $y_{[\mkern-3mu[(r+m){\uparrow}(mr)]\mkern-3mu]}$ with underlying set
$ y_1^{[0{\uparrow}(m{-}2]} y_{[2{\uparrow}r]}$.
Thus we have defined  $y_{[\mkern-3mu[(r+m-1){\uparrow}(mr)]\mkern-3mu]}$
with underlying set $ y_1^{[0{\uparrow}(m{-}2)]} y_{[2{\uparrow}r]}
\cup \{ y_1^{m-1}\}$.
It is easy to see that
    $y_{[\mkern-3mu[(r+m-1){\uparrow}(mr)]\mkern-3mu]} \in
\PU{}(\,\null^1\L^{(m-1)(r-1)+1})$.

Let
\begin{equation*}
Y \coloneq y_{[\mkern-3mu[1{\uparrow}r]\mkern-3mu]} \oplus \mathbf{I}_{m-2}
\oplus y_{[\mkern-3mu[(m+r-1){\uparrow}(mr)]\mkern-3mu]}
= \begin{pmatrix}
y_{[\mkern-3mu[1{\uparrow}r]\mkern-3mu]}&0&0\\
0&\mathbf{I}_{m-2}&0\\
0&0&y_{[\mkern-3mu[(m+r-1){\uparrow}(mr)]\mkern-3mu]}
\end{pmatrix}\,\, \in\,\, \PU{}(\,\null^{m}\L^{\,mr}).
\end{equation*}
We identify $\null^{m} \L^{ mr} = (\Mat{m}(\,\L))^r$, and let
$Y_{[\mkern-3mu[1{\uparrow}r]\mkern-3mu]}$ denote the resulting
partition of $Y$, that is,
\mbox{$Y =  Y_{[\mkern-3mu[1{\uparrow}r]\mkern-3mu]}   \in (\Mat{m}(\L))^r$}.

We then have a well-defined homomorphism $\L  \to
\Mat{m}(\,\L)$ that sends
$y_{[\mkern-3mu[1{\uparrow}r]\mkern-3mu]}$
to~$Y_{[\mkern-3mu[1{\uparrow}r]\mkern-3mu]}$  and  sends
$x_{[\mkern-3mu[1{\uparrow}r]\mkern-3mu]}^{\transp}
(=(y_{[\mkern-3mu[1{\uparrow}r]\mkern-3mu]})^{-1})$ to
$X_{[\mkern-3mu[1{\uparrow}r]\mkern-3mu]}^{\transp}
\coloneq Y^{-1} = Y^{\ast}$.

This homomorphism is nonzero with torsion-free image, and hence is
injective on $\integers$, and hence is
injective, by the following argument of Leavitt~\cite[Theorem 2]{Leavitt2}.
Consider any nonzero element of the kernel.  By multiplying on the
left by a suitable $x$-monomial,
we get a nonzero element in the free $x$-subalgebra.  By multiplying
on the right by a
suitable $y$-monomial, we get a nonzero element of $\integers$, which
is the desired
contradiction.

Let $S$ denote the image of $\operatorname{P}_{\mkern-3mu 1}(\,\L)$,
that is, the
additive monoid that is generated by the multiplicative monoid that
is generated
by $Y_{[1{\uparrow}r]} \cup Y_{[1{\uparrow}r]}^\ast$ in
$\Mat{m}(\,\L)$.  Clearly $S^* = S \subseteq \operatorname{P}_m(\,\L)  $.
It remains to show that $\operatorname{P}_m(\,\L)    \subseteq S$,
for then the injective map $\L  \to
\Mat{m}(\,\L)$ is surjective and the
resulting inverse map carries $\operatorname{P}_{\mkern-3mu m}(\,\L)$ into
$\operatorname{P}_{\mkern-3mu 1}(\,\L)$.

Since $m > r > 2$,\vspace{-3mm}
\allowdisplaybreaks
\begin{align}
Y_1  &= \textstyle\sum\limits_{j\in[1{\uparrow}r]} y_j e_{1,j}
    + \sum\limits_{j\in[(r+1){\uparrow}m]}e_{j-r+1,j}, \label{eq:Y1}\\
Y_2 &=\textstyle\sum\limits_{j\in[1{\uparrow}(r-2)]} \mkern-15mu e_{j+m-r+1,j }
    + \sum\limits_{j\in[(r-1){\uparrow}m]} y_{j+m} e_{m,j}, \label{eq:Y2}\\
    Y_{s} &= \textstyle\sum\limits_{j\in [1{\uparrow}m]} y_{j+(s-1)m}e_{m,j}\,
\text{ for each }    s \in [3{\uparrow}r], \label{eq:Y3}\\
     Y_1^\ast
&\overset{\eqref{eq:Y1}}{=}\textstyle\sum\limits_{j\in[1{\uparrow}r]}
x_j e_{j,1}
    + \sum\limits_{j\in[(r+1){\uparrow}m]}e_{j,j-r+1},\label{eq:Y4}\\
Y_1 Y_1^{\ast} &\overset{\eqref{eq:Y1},\eqref{eq:Y4}}{=}
\textstyle\sum\limits_{j\in[1{\uparrow}(m-r+1)]} e_{j,j},\label{eq:Y5}\\
Y_1e_{j,j}Y_1^{\ast} &\overset{\eqref{eq:Y1},\eqref{eq:Y4}}{=} e_{j-r+1,j-r+1}
\text{ for each }    j \in [(r{+}1){\uparrow}m],\label{eq:Y6}\\
e_{1,1}Y_1e_{j,j} &\overset{\eqref{eq:Y1}}{=} y_je_{1,j}
\text{ for each }    j \in [1{\uparrow}r],\label{eq:Y7}\\
e_{j-r+1,j-r+1}Y_1e_{j,j} &\overset{\eqref{eq:Y1}}{=} e_{j-r+1,j}
\text{ for each }    j \in [(r{+}1){\uparrow}m],\label{eq:Y8}\\
Y_2^\ast &\overset{\eqref{eq:Y2}}{=}\textstyle
\sum\limits_{j\in[1{\uparrow}(r-2)]} \mkern-15mu e_{j ,j+m-r+1}
    + \sum\limits_{j\in[(r-1){\uparrow}m]} x_{j+m} e_{j,m},\label{eq:Y9}\\
Y_2e_{j,j}Y_2^{\ast} &\overset{\eqref{eq:Y2},\eqref{eq:Y9}}{=}
e_{j+m-r+1,j+m-r+1}
\text{ for each }    j \in [1{\uparrow}(r{-}2)],\label{eq:Y10} \\
    e_{j+m-r+1,j+m-r+1}Y_2e_{j,j}  &\overset{\eqref{eq:Y2}}{=} e_{j+m-r+1,j}
\text{ for each }    j \in [1{\uparrow}(r{-}2)],\label{eq:Y11} \\
e_{m,m}Y_2e_{j,j}  &\overset{\eqref{eq:Y2}}{=} y_{j+m}e_{m,j}
\text{ for each }    j \in [(r{-}1){\uparrow}m],\label{eq:Y12} \\
e_{m,m}Y_{s}e_{j,j} &\overset{\eqref{eq:Y3}}{=}   y_{j+(s-1)m}e_{m,j}
\text{ for each }
j \in [1{\uparrow}m],\,\, s \in [3{\uparrow}r].\label{eq:Y13}
\end{align}

\medskip

\noindent\textbf{\ref{thm:aap}.1~Definition. The $m$-cycle $j \mapsto
(j{-}r{+}1)[\mod m]$.}

For each $j \in \integers$, let
$j [\mod  m]$ denote the representative of $j{+}m\integers$ in
$[1{\uparrow}m]$.

Since it is a unit in $\integers_m$, $(r{-}1){+}m\integers$
additively generates a
subgroup of order $m$ in $\integers_m$, and hence shifting down by $r{-}1$
determines an $m$-cycle on $\integers_m$.
Hence  $j \mapsto
(j{-}r{+}1)[\mod m]$ determines an $m$-cycle  on $[1{\uparrow}m]$.
We think of this $m$-cycle as an $m$-gon with two distinguished
sides, $r{-}1 \mapsto m$ and
$r \mapsto 1$. \vspace{-3mm}

\begin{equation}\label{eq:diag}
\begin{matrix}
&m &\mapsto &\cdots &\mapsto &r\\[-2.3mm]
& &&& &{\rotatebox{270}{$\mapsto$}}\\[-4.3mm]
&{\rotatebox{90}{$\mapsto$}}&&& & \\
&r{-}1  &\mapsfrom &\cdots& \mapsfrom   &1
\end{matrix}
\end{equation}

\medskip

\noindent\textbf{\ref{thm:aap}.2~Claim. Both $e_{1,1}$ and $e_{m,m}$
lie in $S$.}

For each $i \in [1{\uparrow}m]$, let
$E_i \coloneq \sum\limits_{j\in[1{\uparrow}i]} e_{j,j}$ and
$E_i' \coloneq \sum\limits_{j\in[(i{+}1){\uparrow}m]} e_{j,j} =
\mathbf{I}_m {-} E_i$.
We shall show   that $E_{i} \in S$  and $E_{i}' \in S$ by letting $i$ travel
around~\eqref{eq:diag} from $m$ to $r{-}1$.
This claim is clear for   $i=m$.

Now suppose that $i \in [1{\uparrow}(r{-}2)] \cup [r{\uparrow}m]$
such that $E_i,\, E'_i \in S$.

If $i \in  [1{\uparrow}(r{-}2)]$, then
\begin{align*}
E_{ i } &= \textstyle\sum\limits_{j\in [1{\uparrow}i] \subseteq
[1{\uparrow}(r{-}2)]}\hskip-20pt e_{j,j} \text{, and }
    (i{-}r{+}1)[\mod m]  = i{+}m{-}r{+}1,\\
    E_{i{+}m{-}r{+}1} &\overset{\eqref{eq:Y10}}{=}
    Y_2E_{i}Y_2^{\ast} + E_{m{-}r{+}1}
    \overset{\eqref{eq:Y5}}{=}     Y_1 Y_1^\ast + Y_2E_{i}Y_2^{\ast} \in S,\\
    E'_{i{+}m{-}r{+}1} &= \mathbf{I}_m  {-} E_{i{+}m{-}r{+}1}
    = \textstyle\sum\limits_{s\in[1{\uparrow}r]} Y_sY_s^*- Y_1
Y_1^\ast - Y_2E_{i
}Y_2^{\ast}
= Y_2E_{ i }'Y_2^{\ast}+ \sum\limits_{s\in[3{\uparrow}r]} Y_sY_s^* \in S.
\end{align*}

If $i \in [r{\uparrow}m]$, then
\begin{align*}
E'_{i}
&= \textstyle\sum\limits_{j\in[i+1,m] \subseteq
[(r{+}1){\uparrow}m]}\hskip-20pt  e_{j,j}  \text{, and }
(i{-}r{+}1)[\mod m]  =  i{-}r{+}1, \\
E'_{ i{-}r{+}1}  &\overset{\eqref{eq:Y6}}{=}  Y_1E'_{i}Y_1^{\ast} + E'_{m-r+1}
    \overset{\eqref{eq:Y5}}{=} Y_1E'_{i}Y_1^{\ast} +  \textstyle
\sum\limits_{s\in
[2{\uparrow}r]} Y_s Y_s^\ast \in S,
\\E_{i{-}r{+}1} &=  \mathbf{I}_m {-}
E'_{i{-}r{+}1}=\textstyle\sum\limits_{s\in[1{\uparrow}r]}
Y_sY_s^*- Y_1E'_{i}Y_1^{\ast}
-   \sum\limits_{s\in [2{\uparrow}r]} Y_s Y_s^\ast=Y_1 E_{i} Y_1^* \in S.
\end{align*}

It now follows by induction on path-length in~\eqref{eq:diag} that,
for each $i \in [1{\uparrow}m]$, $ E_{i} \in S$ and $ E'_{i} \in S$.
Hence, $e_{1,1} = E_1 \in S$ and $e_{m,m} = E_{m-1}' \in S$.
We could have stopped when we had reached whichever came later of $1$
and $m{-}1$.

\medskip

\noindent\textbf{\ref{thm:aap}.3~Claim. Each $e_{j,j}$ lies in $S$.}

Let $j \in [1{\uparrow}m]$.  We shall show that  $e_{j,j} \in S$
by letting $j$ travel  along the top  of~\eqref{eq:diag} from $m$ to $r$
and  along the bottom of~\eqref{eq:diag}  from $1$ to $r{-}1$.   We
have proved the claim for $j= m$ and
$j=1$.  Now suppose that
$j \in [1{\uparrow}(r{-}2)] \cup  [(r{+}1){\uparrow}m]$ such that
$e_{j,j} \in S$.

If  $j \in [1{\uparrow}(r{-}2)]$, then\vspace{-3mm} $$(j{-}r{+}1)[\mod m]  =
j{+}m{-}r{+}1 \text{ and }
    e_{j+m-r+1,j+m-r+1}  \overset{\eqref{eq:Y10}}{=}
Y_2e_{j,j}Y_2^{\ast} \in S.$$

If     $j \in [(r{+}1){\uparrow}m]$, then\vspace{-3mm}
$$(j{-}r{+}1)[\mod m]  = j {-}r{+}1 \text{ and }   e_{j-r+1,j-r+1}
\overset{\eqref{eq:Y6}}{=} Y_1e_{j,j}Y_1^{\ast} \in S.$$

It now follows by induction on path-length in~\eqref{eq:diag} that
$e_{j,j} \in S$ for all $j \in [1{\uparrow}m]$.

\medskip

\noindent \textbf{\ref{thm:aap}.4~Review. The $e_{i,i}Y_s e_{j,j}$ lie in $S$.}

We have now shown that for all $i,\, j \in [1{\uparrow}m]$, and all
$s \in [1{\uparrow}r]$,
$e_{i,i}Y_s e_{j,j}\in S$.
This has the following consequences.
\allowdisplaybreaks
\begin{align}
&\text{For each    $j \in [1{\uparrow}r]$,
$y_je_{1,j}\overset{\eqref{eq:Y7}}{\in} S$.}\label{eq:Y14}
\\&\text{For each $ j \in [(r{+}1){\uparrow}m]$, $ e_{j-r+1,j}
\overset{\eqref{eq:Y8}}{\in} S$.}\label{eq:Y15}
\\&\text{For each $ j \in [1{\uparrow}(r{-}2)]$, $ e_{j+m-r+1,j}
\overset{\eqref{eq:Y11}}{\in} S$.}\label{eq:Y16}
\\&\text{For each $ j \in [(r{-}1){\uparrow}m]$,
$y_{j+m}e_{m,(j+m)[\mod m]}= y_{j+m}e_{m,j}
    \overset{\eqref{eq:Y12}}{\in} S$.}\label{eq:Y17}
\\&\text{For each  $j \in [1{\uparrow}m]$ and $s \in [3{\uparrow}r]$,
$ y_{j+(s-1)m} e_{m,(j+(s-1)m)[\mod m]} =
y_{j+(s-1)m}e_{m,j}\overset{\eqref{eq:Y13}}{\in} S$.}\label{eq:Y18}
\\&\text{For each $k \in [(r{-}1{+}m){\uparrow}(mr)]$, $y_{k} e_{m,k[\mod m]}
\overset{\eqref{eq:Y17},\eqref{eq:Y18}}{\in}S$.}\label{eq:Y19}
\\&y_1^{m-1} e_{m,r-1} = y_{r-1+m}e_{m,(r-1+m)[\mod m]}
\overset{\eqref{eq:Y19}}{\in}S.\label{eq:Y20}
\\&\text{For \mbox{$j \in [1{\uparrow}(m{-}1)]$}, \mbox{$s \in
[2{\uparrow}r]$},
\,\,$y_1^{j-1}y_s e_{m, (s\#j) [\mod m]}  $ }   \label{eq:Y21}
\\&  \hskip6cm
    = y_{(s\#j)+m}e_{m,((s\#j)+m)[\mod
m]}\overset{\eqref{eq:Y19}}{\in}S.\nonumber
\end{align}

\medskip

\noindent \textbf{\ref{thm:aap}.5~Claim. All the $e_{i,j}$ lie in $S$.}

It follows from~\eqref{eq:Y15} and~\eqref{eq:Y16} that for each edge
$j \mapsto j'$ in the top  of diagram~\eqref{eq:diag},
we have $e_{j',j} \in S$.  Since
    $e_{j_1,j_2}e_{j_2,j_3} = e_{j_1,j_3}$, we see that for any
subpath   $j\mapsto\cdots \mapsto j'$  of  the top of
diagram~\eqref{eq:diag}, we have  $e_{j',j} \in S$,
and  $e_{j,j'}  = e_{j',j}^\ast \in S$.
Thus if $j,j'$ are two points on the top of the
diagram~\eqref{eq:diag},   then   $e_{j',j} \in S$.
The same result holds  for the bottom of the diagram~\eqref{eq:diag}.
To obtain  $e_{[1{\uparrow}m]\times [1{\uparrow}m]} \subseteq S$, it
now suffices to
show that $e_{1,m} \in S$.

Recall that for $s \in [2{\uparrow}r]$ and $j \in [1{\uparrow}(m{-}1)]$,
    $s$ and $ s \# j $ lie in the same $\gen{\pi}$-or\-bit.  It is clear that
$\pi$ induces an action modulo $m$, and hence induces a permutation
$\pi_m$ of~$[1{\uparrow}m]$.
Hence $s $ and $ (s \# j) [\mod m]$  lie in the same $\gen{\pi_m}$-or\-bit.
On $[2{\uparrow}(m{-}1)]$, $\pi_m$ acts as \mbox{$i \mapsto
(i{+}r{-}1)[\mod m]$}, while
$1 \mapsto r{-}1$ and $m \mapsto r$. It follows that there are two
$\gen{\pi_m}$-or\-bits  and they are given by the top and the bottom
of the diagram~\eqref{eq:diag}.
     Hence  $e_{s, (s\#j) [\mod m]} \in S$.

In $\L $, $1 = y_1^{m-1}x_1^{m-1}    \textstyle +
    \sum\limits_{j\in[1{\uparrow}(m-1)]} \sum\limits_{s\in[2{\uparrow}r]}
    (y_1^{j-1}y_sx_sx_1^{j-1})$.
Hence, in $\Mat{m}(\,\L)$,
\begin{align*}
e_{1,m} =\,\,& y_1^{m-1}x_1^{m-1}e_{1,m}   \textstyle +
    \sum\limits_{j\in[1{\uparrow}(m-1)]} \sum\limits_{s\in[2{\uparrow}r]}
    (y_1^{j-1}y_sx_s x_1^{j-1}e_{1,m}) \\=\,\,
&(y_1e_{1,1})^{m-1}(e_{1, r-1 })(x_1^{m-1}e_{ r-1 ,m})\\
& \textstyle + \hskip-17.5pt \sum\limits_{j\in[1{\uparrow}(m-1)]}
\sum\limits_{s\in[2{\uparrow}r]}
    (y_1e_{1,1})^{j-1} (y_se_{1,s})(e_{s, (s\#j) [\mod m]}) (x_s
x_1^{j-1}e_{ (s\#j) [\mod m],m}).
\end{align*}
Using~\eqref{eq:Y14},~\eqref{eq:Y20}, and~\eqref{eq:Y21}, and the
fact that $S = S^\ast$,
we see that $e_{1,m} \in S$.

Now $e_{[1{\uparrow}m]\times [1{\uparrow}m]} \subseteq S$.  By~\eqref{eq:Y14},
$y_{[1{\uparrow}r]}e_{[1{\uparrow}m]\times [1{\uparrow}m]}  \subseteq
S$.  Hence $\operatorname{P}_m(\,\L)
\subseteq S$.
This completes the proof.
\end{proof}

\begin{example}  Let us illustrate the proof of Theorem~\ref{thm:aap}
by considering the case $r = 3$ and $m=5$; here $\gcd(m,r{-}1)=
\gcd(5,2)=1$ and $\Mat{5}(\,\L_3) \iso \L_3$.

We find that the cycle decomposition of $\pi$ is $(\ldots, 0, 3,
5,8,10,\ldots )(\ldots, 1,2,4,6,7,9, \ldots).$

Now $y_{m+r-1} \coloneq y_1^{m-1}$, that is, $y_{7} \coloneq y_1^4$.

Now consider $s \in [2{\uparrow}r] = [2{\uparrow}3]$ and $j \in
[1{\uparrow}(m{-1})] = [1{\uparrow}4]$.
We defined $$\{s\#j\} \coloneq
[(1+(r{-}1)j){\uparrow}((r{-}1)(j{+1}))] \cap s^\pi
= [(1+2j){\uparrow}(2j{+}2)]\cap s^\pi.$$
Thus $\{s\#1\} = [3{\uparrow}4] \cap s^\pi$, $\{s\#2\} =
[5{\uparrow}6] \cap s^\pi$, $\{s\#3\}  = [7{\uparrow}8] \cap s^\pi $
and $\{s\#4\} = [9{\uparrow}10] \cap s^\pi$.
For $s = 2$, we are in the set $\{\ldots, 4,6,7,9,\ldots \}$, and for
$s=3$, we are in the set
$\{\ldots, 3,5,8,10, \ldots\}$.
Thus

$2\#1 = 4$, \,\,\, $2\#2 = 6$, \,\,\, $2\#3 = 7$,\,\,\, $2\#4 = 9$,

$3\#1 = 3$, \,\,\, $3\#2 = 5$, \,\,\, $3\#3 = 8$ \,\,\,  $3\#4 = 10$.

\noindent
We define $y_{(s\#j)+m} \coloneq y_1^{j-1} y_s$, that is,
$y_{(s\#j)+5} \coloneq y_1^{j-1} y_s$. Thus

$y_{9} =   y_2$, $y_{11} = y_1 y_2$, $y_{12} = y_1^2y_2$, $y_{14} = y_1^3y_2$,

$y_{8} = y_3$, $y_{10} = y_1y_3$, $y_{13} = y_1^2y_3$ and $y_{15} =
y_1^3y_3$.  Hence

\noindent \mbox{$y_{7} = y_1^4$, $y_{8} = y_3$, $y_{9} =   y_2$,
$y_{10} = y_1y_3$, $y_{11} = y_1 y_2$, $y_{12} = y_1^2y_2$,
    $y_{13} = y_1^2y_3$, $y_{14} {=} y_1^3y_2$,   $y_{15} {=} y_1^3y_3$.}

\noindent Thus $y_{[\mkern-3mu[(r+m-1){\uparrow}(mr)]\mkern-3mu]}=
y_{[\mkern-3mu[7{\uparrow}15]\mkern-3mu]} =
(y_1^{\scriptscriptstyle 4}, y_3,y_2,y_1y_3, y_1y_2, y_1^2y_2,
y_1^2y_3, y_1^3y_2, y_1^3y_3)$.  Now we take
\mbox{$Y \coloneq
y_{[\mkern-3mu[1{\uparrow}r]\mkern-3mu]} \oplus \mathbf{I}_{m-2}
\oplus y_{[\mkern-3mu[(r+m-1){\uparrow}(mr)]\mkern-3mu]}
= y_{[\mkern-3mu[1{\uparrow}3]\mkern-3mu]} \oplus \mathbf{I}_{3}
\oplus y_{[\mkern-3mu[7{\uparrow}15]\mkern-3mu]}$}.
Hence

$$Y=  \left(\begin{smallmatrix}
    y_1&  y_2&\ y_3&0&0&0    &0  &0  &0  &0  &0     &0    &0      &0&0\\
0  &0  &0  & 1  &0  &0   &0  &0  &0  &0  &0     &0    &0      &0&0\\
0  &0  &0  &0  &1  &0   &0  &0  &0  &0  &0     &0    &0      &0&0\\
0  &0  &0  &0  &0  &1   &0  &0  &0  &0  &0     &0    &0      &0&0\\
0  &0  &0  &0  &0  &0   & y_1^{4}& y_3& y_2& y_1y_3
&y_1y_2& y_1^2y_2& y_1^2y_3& y_1^3y_2& y_1^3y_3
\end{smallmatrix}\right).$$

Now we partition $Y$ as

    $Y_1=  \left(\begin{smallmatrix}
    y_1&  y_2&  y_3&0&0 \\
0  &0  &0  & 1  &0  \\
0  &0  &0  &0  &1   \\
0  &0  &0  &0  &0    \\
0  &0  &0  &0  &0
\end{smallmatrix}\right),$
$Y_2=  \left(\begin{smallmatrix}
    0    &0  &0  &0  &0  \\
      0   &0  &0  &0  &0  \\
      0  &0  &0  &0  &0    \\
    1   &0  &0  &0  &0  \\
    0   & y_1^{4}& y_3& y_2& y_1y_3
\end{smallmatrix}\right),$
$Y_3=  \left(\begin{smallmatrix}
    0     &0    &0      &0&0\\
0     &0    &0      &0&0\\
0     &0    &0      &0&0\\
0     &0    &0      &0&0\\
y_1y_2& y_1^2y_2& y_1^2y_3& y_1^3y_2& y_1^3y_3
\end{smallmatrix}\right).$

\medskip

We let $X_i \coloneq Y_i^\ast$.

Now

\noindent $Y_1 X_1  = \left(\begin{smallmatrix}
    y_1&  y_2&  y_3&0&0 \\
0  &0  &0  & 1  &0  \\
0  &0  &0  &0  &1   \\
0  &0  &0  &0  &0    \\
0  &0  &0  &0  &0
\end{smallmatrix}\right) \left(\begin{smallmatrix}
    x_1&  0&  0&0&0 \\
x_2  &0  &0  & 0  &0  \\
x_3 &0  &0  &0  &0   \\
0  &1  &0  &0  &0    \\
0  &0  &1  &0  &0
\end{smallmatrix}\right) = \left(\begin{smallmatrix}
    1&  0&  0&0&0 \\
0  &1  &0  & 0  &0  \\
0  &0  &1  &0  &0   \\
0  &0  &0  &0  &0    \\
0  &0  &0  &0  &0
\end{smallmatrix}\right)= E_3.$

\medskip

\noindent Thus, $E_3 = Y_1  X_1$.  Hence, $E_3' = Y_2  X_2  + Y_3  X_3$.

\medskip

Now
\noindent $Y_1 E_3' X_1 =$ \\ $ \left(\begin{smallmatrix}
    y_1&  y_2&  y_3&0&0 \\
0  &0  &0  & 1  &0  \\
0  &0  &0  &0  &1   \\
0  &0  &0  &0  &0    \\
0  &0  &0  &0  &0
\end{smallmatrix}\right)
\left(\begin{smallmatrix}
    0&  0&  0&0&0 \\
0  &0  &0  & 0  &0  \\
0  &0  &0  &0  &0   \\
0  &0  &0  &1  &0    \\
0  &0  &0  &0  &1
\end{smallmatrix}\right)
    \left(\begin{smallmatrix}
    x_1&  0&  0&0&0 \\
x_2  &0  &0  & 0  &0  \\
x_3 &0  &0  &0  &0   \\
0  &1  &0  &0  &0    \\
0  &0  &1  &0  &0
\end{smallmatrix}\right) =
    \left(\begin{smallmatrix}
0&  0&  0&0&0 \\
0  &0  &0  & 1  &0  \\
0  &0  &0  &0  &1   \\
0  &0  &0  &0  &0    \\
0  &0  &0  &0  &0
\end{smallmatrix}\right)
    \left(\begin{smallmatrix}
    x_1&  0&  0&0&0 \\
x_2  &0  &0  & 0  &0  \\
x_3 &0  &0  &0  &0   \\
0  &1  &0  &0  &0    \\
0  &0  &1  &0  &0
\end{smallmatrix}\right) = \left(\begin{smallmatrix}
    0&  0&  0&0&0 \\
0  &1  &0  & 0  &0  \\
0  &0  &1  &0  &0   \\
0  &0  &0  &0  &0    \\
0  &0  &0  &0  &0
\end{smallmatrix}\right)= E_1'{-}E_3'.$

\medskip

\noindent Hence
$E_1' =  Y_1 E_3' X_1+E_3'   = Y_1 Y_2 X_2  X_1  + Y_1 Y_3 X_3  X_1
+ Y_2 X_2  + Y_3 X_3$ which we abbreviate to
$E_1' =   Y_{1,2}X_{2,1} + Y_{1,3} X_{3,1} + Y_2 X_2  + Y_3 X_3$.
Hence $E_1 = Y_{1,1} X_{1,1}$.

\medskip

Similar straightforward calculations show that
    $Y_2 E_1 X_2  =  E_4{-}E_3.$  Hence

\noindent  $E_4 = E_3 + Y_2 E_1 X_2  = Y_1 X_1 + Y_{2,1,1} X_{1,1,2}$. Hence

\noindent  $E_4' = Y_{2,1,2}X_{2,1,2} + Y_{2,1,3} X_{3,1,2} + Y_{2,2} X_{2,2}
    + Y_{2,3} X_{3,2} +  Y_3X_3$.

\medskip

     Similarly, $Y_1 E_4' X_1  = E_2' - E_3'.$  Hence
    $E_2' =  Y_1 E_4' X_1 +E_3' $.  Hence

\noindent $E_2' =  Y_{1,2,1,2} X_{2,1,2,1} + Y_{1,2,1,3}X_{3,1,2,1}+
Y_{1,2,2}X_{2,2,1} + Y_{1,2,3}X_{3,2,1} + Y_{1,3}X_{3,1} + Y_2 X_2 +
Y_3 X_3$.  Hence

\noindent  $ E_2 = Y_{1,1}X_{1,1} + Y_{1,2,1,1} X_{1,1,2,1}$.

Now

\noindent $e_{1,1} = E_1   = Y_{1,1} X_{1,1} $.

\noindent $e_{4,4} = Y_2 e_{1,1} X_2  =    Y_{2,1,1}X_{1,1,2} $.

\noindent $e_{2,2} =  Y_1 e_{4,4} X_1  =  Y_{1,2,1,1}   X_{1,1,2,1}$.

\noindent $e_{5,5} =   E_4' =
Y_{2,1,2}X_{2,1,2}+Y_{2,1,3}X_{3,1,2}+Y_{2,2}X_{2,2}+Y_{2,3}X_{3,2}+Y_3X_3$.

\noindent  $e_{3,3} =  Y_1 e_{5,5} X_1  =
Y_{1,2,1,2}X_{2,1,2,1}+Y_{1,2,1,3}X_{3,1,2,1}+
Y_{1,2,2}X_{2,2,1}+Y_{1,2,3}X_{3,2,1}+Y_{1,3}X_{3,1}$.

The interested reader can calculate the expressions for the remaining
$e_{i,j}$.
\end{example}

\section{The Abrams-\'Anh-Pardo Theorem}\label{sec:7}

The following is a straightforward consequence of the
Chinese remainder theorem; the earliest mention of it that we have found
   is~\cite[p.\ 466, line 9]{Kaplansky}.

\begin{lemma}\label{lem:kap}
Let $m_1$, $m_2$, $s   \in  \integers$.  If  $\gcd(m_1,s) = \gcd(m_2,s)$, then
there exists $u\in \integers$ such that\, $um_1  \equiv m_2
\,\,{\normalfont\mod} s$ and $\gcd(u,s) = 1$.
\end{lemma}

\begin{proof}
Note first that if $s=0$ then $m_1=\pm m_2$ and we can take $u=\pm
1$. Thus we may assume $s\neq 0$.

Let $g \coloneq \gcd(m_1,s) = \gcd(m_2,s)$. There exist $n_1,\,n_2
\in \integers$ such that $ n_1 g  = m_1 $ and $n_2 g  = m_2$.
  By Euclid's lemma, there exist   $k_1,\, k_2 \in \integers$ such
that $m_1k_1 \equiv  g  \,\mod s$
and $m_2k_2 \equiv  g  \,\mod s$.

Let $R \coloneq \integers_s$, $a \coloneq m_1
+ s\integers$,  $b \coloneq m_2 + s\integers$, $c \coloneq n_1k_2 + s
\integers$, $d \coloneq n_2k_1 + s\integers$.

We have $a$, $b$, $c$, $d \in R$ such that
$ad = b$ and $bc = a$,  and it suffices to find some unit $x = u+ s
\integers \in R$ such that $ax= b$.

Eliminating $b$, we then have $a$, $c$, $d \in R$ such that
$a(1{-}cd) = 0$, and it suffices to find
some unit $x \in R$ such that $ax = ad$.

If $R = \integers_{p^m}$ where $p$ is a prime number and $m \ge 1$,
then either $a=0$ and here we can take $x=1$ as a solution, or $a
\ne 0$ and then $1{-}cd$ is a zerodivisor, hence $(1{-}cd)^m= 0$,
hence $ 1-(1{-}cd) $ is a unit, hence $cd  $ is a unit, hence $d$ is
a unit, hence $x=d$ is a solution.

  By the Chinese remainder
theorem,
  $R$ is a direct product of  a finite number of rings of the form
$\integers_{p^m}$ where $p$ is a prime number and $m \ge 1$.
By the preceding paragraph, we can find a suitable unit in each of
these factors, and then form a suitable
unit in $R$.  This completes the proof.
\end{proof}

The following is also well known.

\begin{lemma}\label{lem:first} Let  $r \in [2{\uparrow} \infty[\,$
and  \mbox{$m,\,m' \in
[1{\uparrow}\infty[\,$}.  If $m' \equiv m  \,\,\normalfont{\mod}
(r{-}1)$, then $\Mat{m'}(\,\L_r)$ and $\Mat{m}(\,\L_r)$ are
isomorphic as   partially ordered rings with involution.
\end{lemma}

\begin{proof}  Let $\L \coloneq \L_r$.

Consider first the case where there exists some
$Y \in  \,\PU{}(\,\null^{m'} \L^{m})$.
We then have a map
$\Mat{m'}(\,\L)  \to  \Mat{m}( \L )$,  $\underset{m'\times
m'}{M} \mapsto
\underset{ m  \times m'}{Y^*}\,\,\underset{m'\times
m'}{M}\,\,\underset{m'\times m}{Y}$,
and it is easily seen to be a homomorphism of partially order rings
with involution.
Using $Y^\ast$
in place of $Y$, we get a map in
the reverse direction, and the two maps are mutually inverse.  Thus
it suffices to  prove that  $\PU{}(\mkern1mu\null^{m} \L^{m+r-1})$
is nonempty.
Now $y_{[\mkern-3mu[1{\uparrow}r]\mkern-3mu]} \in
\PU{}(\mkern1mu\null^{1} \L^{r})$
and $\mathbf{I}_{m-1} \in \PU{m-1}(\,\L) =
\PU{}(\mkern1mu\null^{m-1} \L^{m-1})$.  Here,
we have the diagonal sum
\begin{equation*}  \underset{1\times
r}{y_{[\mkern-3mu[1{\uparrow}r]\mkern-3mu]}}
    \oplus  \underset{(m-1)\times (m-1)}{\mathbf{I}_{m-1}}
\coloneq \begin{pmatrix}
y_{[\mkern-3mu[1{\uparrow}r]\mkern-3mu]}&0 \\
0&\mathbf{I}_{m-1}
\end{pmatrix}\,\, \in\,\,
\PU{}(\,\null^{m } \L^{m+r-1}).
\end{equation*}
This  completes the proof.
\end{proof}

\begin{AAP2}\label{final}
Let $r \in [2{\uparrow}\infty]$ and $m,\,m',\,t \in [1{\uparrow}\infty[\,$.
If $\gcd(m',r{-}1) = \gcd(m,r{-}1)$, then
$\Mat{m'} (\,\L_r^{\otimes t})$ and
$\Mat{m} (\,\L_r^{\otimes t})$
are isomorphic as partially ordered rings with involution,
and, hence,  $\PU{m'}\mkern1mu (\,\L_r^{\otimes t}) \iso
\PU{m}\mkern1mu(\,\L_r^{\otimes t})$.
\end{AAP2}

\begin{proof} For the purposes of this proof,  let us  write
$\underset{\text{porwi}}{\iso}$ to indicate ``isomorphic as partially
ordered rings with involution''.

We claim that
$\textstyle\Mat{m'}(\L_{r})\underset{\text{porwi}}{\iso}
\Mat{m}(\L_{r})$. If $r=2$, this holds   by  Lemma~\ref{lem:first};
thus we may assume that $r \ge 3$.
   By Lemma~\ref{lem:kap}, there exists  $u\in \integers$ such that
$m'u  \equiv m \,\,{\normalfont\mod} (r{-}1)$ and \mbox{$\gcd(u,r{-}1) = 1$}.
By adding some multiple of $r{-}1$ to $u$, we may further assume that $u > r$.
By Theorem~\ref{thm:aap},
$ \L_{r}\underset{\text{porwi}}{\iso}
\Mat{u}\mkern-1mu(\integers)\otimes_\integers \L_{r}$.
It is well known that $\Mat{m'}(\integers)\otimes_\integers
\Mat{u}\mkern-1mu(\integers)
\underset{\text{porwi}}{\iso} \Mat{m'u}(\integers)$.
By  Lemma~\ref{lem:first}, $\Mat{m'u}(\integers)\otimes_\integers  \L_{r}
\underset{\text{porwi}}{\iso}
\Mat{m}(\integers)\otimes_\integers  \L_{r}$.
It then follows that
$$\textstyle\Mat{m'}(\integers)\otimes_\integers
\L_{r}\underset{\text{porwi}}{\iso}
\Mat{m'}(\integers)\otimes_\integers \Mat{u}\mkern-1mu(\integers)
\otimes_\integers \L_{r} \underset{\text{porwi}}{\iso}
\Mat{m'u}(\integers)\otimes_\integers  \L_{r}
\underset{\text{porwi}}{\iso}
\Mat{m}(\integers)\otimes_\integers \L_{r},$$
and the claim is proved.

Now, for $t \ge 2$, applying $(-)\otimes_{\integers}\L_{r}^{\otimes
(t-1)}$ gives the desired result.
\end{proof}

\section{The Brin-Higman-Thompson group   $tV_{r,m}$ is
$\PU{m}(\L_r^{\otimes t})$}\label{sec:3}

We now consider the Brin-Higman-Thompson groups.  To lead into the definition
gradually, we consider first the
Higman-Thompson groups.

\begin{definitions}\label{defs:HT} Let
    $r \in [2{\uparrow}\infty[\,$ and
$m \in [1{\uparrow}\infty[\,$.
We now recall one of the constructions of  the Higman-Thompson group
$V_{r,m}$ from~\cite{Higman}; see also~\cite{Scott}.

Let $\L \coloneq \L_r$.
By Leavitt's normal-form result, the multiplicative submonoid
$\langle \langle y_{[1{\uparrow}r]}
\rangle\rangle$ of $\L$
is the free monoid on $ y_{[1{\uparrow}r]}$.
    We view the Cartesian product  $  e_{[1{\uparrow}m]\times\{1\}}\times
\langle \langle
y_{[1{\uparrow}r]} \rangle\rangle  $
as the product
    \mbox{$e_{[1{\uparrow}m]\times\{1\}}\langle\langle
y_{[1{\uparrow}r]}\rangle\rangle
\subseteq\Mat{m}(\,\L)$}.

Let $A$ be any  finite subset  of $e_{[1{\uparrow}m]\times\{1\}}
\langle \langle
y_{[1{\uparrow}r]} \rangle\rangle$.
For any $a \in A$,  the $a$th  \textit{expansion of $A$}
is
$$\partial_a(A) \coloneq (A \setminus\{a\}) \cup a y_{[1{\uparrow}r]}
\subseteq
     e_{[1{\uparrow}m]\times\{1\}}  \langle \langle y_{[1{\uparrow}r]}
\rangle\rangle.$$

Let $\mathfrak{B}_m$ denote the
smallest set of (finite) subsets of $e_{[1{\uparrow}m]\times\{1\}}
\langle \langle
y_{[1{\uparrow}r]} \rangle\rangle $
such that   $e_{[1{\uparrow}m]\times\{1\}} \in \mathfrak{B}_m$  and
$\mathfrak{B}_m$ is closed under
taking expansions, that
is,   whenever $A \in \mathfrak{B}_m$ and $a \in A$, then
$\partial_a(A) \in \mathfrak{B}_m$.
An element of   $\mathfrak{B}_m$ is called a  \textit{basis}.

For any $A \in \mathfrak{B}_m$, we can apply suitable expansions and
arrive at an element
\mbox{$B  \in \mathfrak{B}_m$} whose elements all have the same length, and
then we have \emph{all}
of the elements of  $e_{[1{\uparrow}m]\times\{1\}}\langle \langle
y_{[1{\uparrow}r]} \rangle\rangle  $
of this  length. Any such $B$  is called a \textit{homogeneous}
element  of $\mathfrak{B}_m$.

We now consider the set of maps that are bijections between elements of
$\mathfrak{B}_m$, $$\Phi \coloneq \{ \phi : A \to B,\, a \mapsto
a^\phi \mid A, B
    \in \mathfrak{B}_m, \,\, \phi \,\, \text{bijective}\}.$$
We shall construct  $V_{r,m}$ using equivalence classes in $\Phi$.

Suppose that $ A \xrightarrow{\phi} B$ is an element of $\Phi$,
and that $a \in A$, and let $b \coloneq a^\phi$.  We define
$\partial_a(\phi): \partial_a(A) \to  \partial_{b}(B)$
in the natural way, that is, $\partial_a(\phi)$ acts as $\phi$ for the
bijection \mbox{$A \setminus \{a\} \to B \setminus \{b\}$},
and sends   $ay_s$ to $by_s$ for each $s \in [1{\uparrow}r]$.
We call $\partial_a(\phi)$  the $a$th  \textit{expansion} of $\phi$.

We define the set  $V_{r,m}$ to consist of the equivalence classes in
$\Phi$ obtained by identifying
each element of $\Phi$ with all of its expansions.

We define a binary operation on $V_{r,m}$ as follows.
For any $\phi$, $\psi \in \Phi$, we can take successive
expansions of $\phi^{-1}$ and $\psi$ until they have homogeneous
domains of the same length,
   in particular until they have the same domain.
We may then compose $\phi\, \psi$.  We then obtain a well-defined
binary operation on $V_{r,m}$.
This concludes the  definition of
\textit{the Higman-Thompson group} $V_{r,m}$.

Let us mention some subgroups of $V_{r,m}$.
We give
   $e_{[1{\uparrow}m]\times\{1\}} \langle \langle
y_{[1{\uparrow}r]}\rangle\rangle$ the lexicographic ordering.
If $A$, $B \in \mathfrak{B}_m$ have the same size and $A
\xrightarrow{\phi} B$ is the unique bijective map that
respects the induced orderings, then all the expansions of $A
\xrightarrow{\phi} B$ will also respect the
induced orderings. The set of elements of $V_{r,m}$
represented by order-preserving maps form a subgroup of $V_{r,m}$,
denoted $F_{r,m}$.
Similarly, we can allow  $A \xrightarrow{\phi} B$
to be one of the maps that respects the induced orderings cyclically.
We then get the subgroup
$T_{r,m}$ of $V_{r,m}$ that contains $F_{r,m}$; see~\cite{Higman} or
\cite{Brown}.
Here, $F_{2,1}$ and $T_{2,1}$ are Thompson's group $F$ and $T$, respectively.
\end{definitions}

\begin{definitions}\label{defs:BHT}  Let $m,\,t
\in[1{\uparrow}\infty[$ and $r \in [2{\uparrow}\infty[\,$.
We now define the Brin-Higman\d1Thompson group  $tV_{r,m}$ along
the same lines as in the above definition of the Higman\d1Thompson groups.

Let $\L \coloneq \L_r$.
For $\ell\in[1{\uparrow}t]$, $k\in[1{\uparrow r}]$, we define
  $y_{\ell,k}:= 1^{\otimes (\ell-1)}\otimes  y_k \otimes 1^{\otimes
(t-\ell)}\in\L^{\otimes t}$ and $x_{\ell,k}:= y_{\ell,k}^\ast =
1^{\otimes (\ell-1)}\otimes  x_k
\otimes 1^{\otimes (t-\ell)}\in\L^{\otimes t}.$
We view the Cartesian product  $$  e_{[1{\uparrow}m]\times\{1\}}\times
\langle\langle y_{[1\uparrow t]\times[1\uparrow r]}\rangle\rangle $$
as the product
$ e_{[1{\uparrow}m]\times\{1\}}\langle\langle y_{[1\uparrow
t]\times[1\uparrow r]}\rangle\rangle
\subseteq\Mat{m}(\,\L^{\otimes t}).$

We consider $t$ different kinds of \textit{expansions} on a finite
subset  $A$ of
$e_{[1{\uparrow}m]\times\{1\}} \langle\langle y_{[1\uparrow
t]\times[1\uparrow r]}\rangle\rangle$ as follows. For
each
$\ell\in[1{\uparrow}t]$, $a \in A$, let
$$\partial_{\ell,a}(A) \coloneq  A \setminus \{a\} \cup a
y_{\{\ell\}\times [1\uparrow r]} .$$
Let $\mathfrak{B}^{(t)}_m$
be the smallest set
of subsets of
$e_{[1{\uparrow}m]\times\{1\}} \langle\langle y_{[1\uparrow
t]\times[1\uparrow r]}\rangle\rangle$
such that  $e_{[1{\uparrow}m]\times\{1\}} \in \mathfrak{B}^{(t)}_m$
and $\mathfrak{B}^{(t)}_m$ is closed under taking expansions of all
kinds. The elements of
$\mathfrak{B}^{(t)}_m$ are
called \textit{bases}.

A subset $A$ of  $ e_{[1{\uparrow}m]\times\{1\}} \langle\langle
y_{[1\uparrow t]\times[1\uparrow
r]}\rangle\rangle$  is said to be \textit{unitary} if it satisfies
    \mbox{$\sum\limits_{a\in A} (a \cdot a^\ast) = \textbf{I}_m$}, and,
for all $a,\,b \in A$, if $a \ne b$, then $a^\ast {\,\cdot\,} b = 0$
(and thus $a$ is not a prefix of~$b$).  It is not difficult to show that every
expansion of a unitary set is unitary.
Notice that the question of multiplicity does not arise since, in a
unitary set,
no element is a prefix of another. Since
$e_{[1{\uparrow}m]\times\{1\}}$ is a unitary set, we see that
every basis is unitary.

Each $b \in e_{[1{\uparrow}m]\times\{1\}} \langle\langle
y_{[1\uparrow t]\times[1\uparrow
r]}\rangle\rangle$ can be expressed uniquely as a product $b = e_{i1}
b_1 \cdots b_t$,
where each $b_\ell$ lies in $\langle\langle
y_{\{\ell\}\times[1\uparrow r]}\rangle\rangle$,
for each \mbox{$\ell \in [1{\uparrow}t]$}.
The length of $b_\ell$ is called the \textit{$\ell$-length} of $b$.

A finite subset $A$ of  $ e_{[1{\uparrow}m]\times\{1\}}
\langle\langle y_{[1\uparrow t]\times[1\uparrow
r]}\rangle\rangle$  is
\textit{multi-homogeneous} if, for each \mbox{$\ell \in
[1{\uparrow}t]$}, all the elements of $A$
have the same $\ell$-length.
Clearly, any finite subset of $ e_{[1{\uparrow}m]\times\{1\}}
\langle\langle y_{[1\uparrow t]\times[1\uparrow
r]}\rangle\rangle$  can be expanded
to a multi-homogeneous subset.  In particular, any
basis can be expanded to a multi-homogeneous basis, which will then
  have all the elements
that have the specified $\ell$-length,  for each $\ell$.
(See also \cite{KochMPNuc} Lemma 3.2.)

If $B$ is a   multi-homogeneous unitary set, then $B$ lies in a
unique multi-homogeneous basis $C$.  If $B \ne C$, then, with respect
to the partial order on
   $\Mat{m}(\L^{\otimes t})$, we would have
   \mbox{$\textbf{I}_m = \sum\limits_{b\in B} (b {\,\cdot\,} b^\ast)  <
\sum\limits_{c\in C} (c{\,\cdot\,} c^\ast) = \textbf{I}_m$},
which is a contradiction.  Thus, $B=C$.  Hence, each
multi-homogeneous unitary set is a basis.
Hence, each unitary set can be expanded to a multi-homogeneous basis.

We now consider the set of maps that are bijections between elements of
$\mathfrak{B}_m^{(t)}$, $$\Phi \coloneq \{  A \xrightarrow{\phi} B \mid A, B
    \in \mathfrak{B}_m^{(t)}, \,\, \phi \,\, \text{bijective}\}.$$
We  construct  the Brin-Higman-Thompson group  $tV_{r,m}$ as
the set of equivalence classes in $\Phi$ in the same way that we defined
   the Higman-Thompson group $V_{r,m}$  in
Definitions~\ref{defs:HT}.

For $t \ge 2$, the symbols $tF_{r,m}$ and  $tT_{r,m}$ have not been
assigned definitions;
Brin~\cite[Remark~4.9]{Brin} discusses  his unsuccessful efforts to define a
  $2F_{2,1}$  with desirable properties.
\end{definitions}

\begin{remarks}\label{rems:unitary}  If $t=2$, then every unitary set
is a basis.
To see this, suppose that $B$ is unitary.
It suffices to consider the case $m=1$.
Recall that each $b \in B$ has a factorization $b =   b_1 b_2 = b_2 b_1$
with $b_i\in\langle\langle y_{\{i\}\times[1\uparrow
r]}\rangle\rangle$ for $i=1,2$.
Consider first the case where,
  for some $b\in B$, we have $b_1 = 1$ and $b_2 \ne 1$.  Here, for
each $c \in B$, if $c \ne b$, then
$c_2{\,\cdot\,} b^\ast= 0$ and $c_2 \ne 1$.  Then $B$ is a disjoint union of
$y_{2,k} B_k$ for each $k \in [1{\uparrow}r]$.
    Each $B_k$ is unitary,
and by induction is a basis.  Hence $B$ is a basis.  In the remaining case,
for each $b\in B$, we have $b_1 \ne 1$, and then $B$ is a disjoint union of
$y_{1,k} B_k$ for each $k \in [1{\uparrow}r]$, and,  by the same
argument, $B$ is a basis.

Similarly, if $t=1$, then each unitary set is a basis.

For $t =3$, $m=1$, and $r=2$,
$
\{y_{2,1}y_{3,1},\,y_{1,1}y_{2,1}y_{3,2},\,y_{1,1}y_{2,2},\,y_{1,2}y_{2,2}y_{3,1},\,y_{1,2}y_{3,2}\}$
is a unitary set that is not a basis.
\end{remarks}

We now come to our main result.
In \cite{Pardo}, Pardo found this result for Higman-Thompson groups,
i.e., in the case $t=1$.

\begin{theorem} \label{thm:Hig} Let
    $r \in [2{\uparrow}\infty[\,$ and
$m,\,t \in [1{\uparrow}\infty[\,$.
Then   $ \PU{m}(\,\L_r^{\otimes t})$ is   isomorphic
to the Brin-Higman-Thompson group  $tV_{r,m}$.
\end{theorem}

\begin{proof} We use the notation of Definitions~\ref{defs:BHT}.

For each $(A \xrightarrow{\phi} B) \in \Phi$, we define
   $(A \xrightarrow{\phi} B)^\alpha \coloneq \sum\limits_{a \in A}
(a \cdot (a^{\phi})^\ast) \in \operatorname{P}_m(\L^{\otimes t})$.
It is readily verified that $\alpha$ has the same value on
all the expansions of $(A \xrightarrow{\phi} B)$.
Also, $$\textstyle(B \xrightarrow{\phi^{-1}} A)^\alpha = \sum\limits_{a \in A}
(a^{\phi}  \cdot  a^\ast) = ((A \xrightarrow{\phi} B)^\alpha)^{\ast}.$$
Thus we have a well-defined map of sets
   $\alpha \colon t V_{r,m} \to \operatorname{P}_m(\L^{\otimes t})$.
It is a morphism of multiplicative monoids, since the identity
maps to the identity, and, for any
   $(A \xrightarrow{\phi} B)$,
   $(B \xrightarrow{\psi} C) \in \Phi$,
\begin{align*}
\textstyle (A \xrightarrow{\phi} B)^\alpha \cdot (B
\xrightarrow{\psi} C)^\alpha &=\textstyle
\sum\limits_{a\in A} (a \cdot (a^{\phi})^\ast) \cdot
\sum\limits_{b\in B} (b \cdot (b^{\psi})^\ast)
= \sum\limits_{a\in A} (a \cdot (a^{\phi})^\ast) \cdot   (a^{\phi}
\cdot ((a^{\phi})^{\psi})^\ast)
\\&\textstyle=  \sum\limits_{a\in A}    (a \cdot    (a^{\phi\, \psi})^\ast)
=  (A \xrightarrow{\phi\,\psi} C)^\alpha.
\end{align*}
In particular, $(B \xrightarrow{\phi^{-1}} A)^\alpha =
((A \xrightarrow{\phi} B)^\alpha)^{-1}$; as we have already seen that
$(B \xrightarrow{\phi^{-1}} A)^\alpha =
((A \xrightarrow{\phi} B)^\alpha)^{\ast}$, we see that
$(A \xrightarrow{\phi} B)^\alpha \in \PU{m}(\L^{\otimes t})$.
In summary, we have a well-defined homomorphism
   $\alpha \colon t V_{r,m} \to \PU{m}(\L^{\otimes t})$ which sends the
equivalence class of $(A \xrightarrow{\phi} B)$ to
$\sum\limits_{a \in A}
(a \cdot (a^{\phi})^\ast)$.

We  next prove surjectivity of $\alpha \colon tV_{r,m} \to
\PU{m}(\,\L^{\otimes t})$.

Consider an arbitrary  \mbox{$u \in \PU{m}(\,\L^{\otimes t})$}.
Since $u\in\operatorname{P}_{m}(\,\L^{\otimes t}) $, we have an
expression of $u$
as a sum of elements of the form
$ e_{i,j}{\,\cdot\,} w {\,\cdot\,}z^\ast$ with \mbox{$w$, $z \in
\langle\langle y_{[1\uparrow
t]\times[1{\uparrow}r]}\rangle\rangle$}; notice that \mbox{$ e_{i,j}
{\,\cdot\,}w {\,\cdot\,}z^\ast
   = (e_{i,1}{\,\cdot\,}w){\,\cdot\,}(e_{j,1}{\,\cdot\,}z)^\ast $}.
    By repeatedly inserting $\sum\limits_{k\in[1{\uparrow}r]}
(y_{\ell,k}{\,\cdot\,} x_{\ell,k}) \,\,(=1)$
between suitable
    $w$ and $z^\ast$, we can arrange for  all the $w$s to have the same
$\ell$-length, for each $\ell \in [1{\uparrow}t]$, and
obtain an expression
\mbox{$u = \sum\limits_{a\in A}  a  \cdot  p_{a}^\ast $},
where $A \in \mathfrak{B}_m^{(t)}$,  multi-homogeneous, and, for
each $a \in A$,  $p_a$ is a
     sum of elements from
$e_{[1{\uparrow}m]\times\{1\}}\langle\langle y_{[1\uparrow
t]\times[1{\uparrow}r]}\rangle\rangle$. It is not difficult to see
that if $p_a$ has at least two summands then
   $e_{1,1} <  (p_a^\ast) {\,\cdot\,} (p_{a})$.
   Since it is a basis,  $A$ is a unitary set.
Let \mbox{$B \coloneq \{p_a \mid a \in A\}$}, and let $  A
\xrightarrow{\phi} B$ be given by
$a \mapsto p_a$.  We shall show that $B$ is a unitary set and that
$\phi$ is injective.
   Let $a,\,a' \in A$.
Since $u$ is a unitary matrix, $u{\,\cdot\,}u^\ast = \mathbf{I}_m$, and, hence,
$$ (p_a^\ast) {\,\cdot\,} (p_{a'}) =
(a^\ast  {\,\cdot\,}  u){\,\cdot\,} (u^\ast {\,\cdot\,} a')
=  a^\ast {\,\cdot\,} \mathbf{I}_m {\,\cdot\,} a' = a^\ast {\,\cdot\,} a'
=  e_{1,1} \delta_{a,a'}.$$
In particular, $(p_a^\ast) {\,\cdot\,} (p_{a}) = e_{1,1}$,
and we see that $p_a$ has exactly one summand, that is,
\mbox{$B \subseteq    e_{[1{\uparrow}m]\times\{1\}}\langle\langle y_{[1\uparrow
t]\times[1{\uparrow}r]}\rangle\rangle$}.  If
\mbox{$a' \ne a$}, we have
$(p_a^\ast) {\,\cdot\,} (p_{a'})=0$; in particular, $\phi$ is
injective and, hence, bijective.
We also have
$$\textstyle
   \sum\limits_{b\in B}  (b {\,\cdot\,}    b^\ast)
=\sum\limits_{a\in A}  (p_a {\,\cdot\,}    p_{a}^\ast)
=
(\sum\limits_{a\in A}  p_a {\,\cdot\,}  {a}^\ast) {\,\cdot\,}
(\sum\limits_{a'\in A}  a'  {\,\cdot\,}  p_{a'}^\ast)
= u^\ast {\,\cdot\,} u
=
\mathbf{I}_m.
$$
Thus $B$ is a unitary set.  We may expand $B$ to a multi-homogeneous set $B'$,
and   $B'$ is again a unitary set, and  is then a basis; see
Definitions~\ref{defs:BHT}.
By considering the corresponding expansion of $\phi^{-1}$ we get an
expansion $A'$ of $A$, which is again a basis, and an expansion
   $ A' \xrightarrow{\phi'}  B'$ of $\phi$.
Then $(A' \xrightarrow{\phi'}  B') \in \Phi$  and
$\sum \limits_{a'\in A'} (a' {\,\cdot\,} (a'^{\phi'})^\ast) =
\sum \limits_{a \in A } (a {\,\cdot\,} (a^{\phi})^\ast) =
u$\vspace{-3mm}.  This proves that
\mbox{$\alpha \colon tV_{r,m} \to
\PU{m}(\,\L^{\otimes t})$} is surjective.

It remains to show that $\alpha \colon tV_{r,m} \to
\PU{m}(\,\L^{\otimes t})$ is injective. Suppose   that
$(A \xrightarrow{\phi} B) \in \Phi$ and $\sum\limits_{a\in A} (a
{\,\cdot\,} (a^\phi)^\ast) = \mathbf{I}_m$.
For each $a \in A$,  right multiplying the latter equation by
$a^\phi$ gives $a=  a^\phi$.
This proves that $\alpha \colon tV_{r,m} \to
\PU{m}(\,\L^{\otimes t})$ is injective.
\end{proof}

\section{Higman's proof of invariance of $r$ and
$\gcd(m,r{-}1)$}\label{sec:hig}

We use the notation of Definitions~\ref{defs:BHT}.

\begin{background}\label{back:hig}  Here we quickly review the main
points of Higman's analysis of conjugacy
   classes of finite subgroups of $tV_{r,m}$.
The arguments are easily adapted from the articles \cite{Higman},
\cite{MPNuc}, both of which are set in
a broader framework where two bases need not have a common ancestor.

For each $n \in [1{\uparrow}\infty[\,$,  there exists some $B \in
\mathfrak{B}_{m}^{(t)}$
with $\abs{B} = n$   if and only if
\mbox{$n \equiv m \text{ mod } (r{-}1)$}.

   By working with minimal common expansions, one can show that
each finite subgroup $H$ of $t V_{r,m}$ permutes the elements of some
$B \in \mathfrak{B}_{m}^{(t)}$; moreover,
the conjugacy class of $H$ in $t V_{r,m}$ is then determined by the
decomposition of $B$ into
$H$-orbits modulo identifying expansions of entire $H$-orbits. Here we will be
counting the number of isomorphic copies of an orbit
modulo $r{-}1$ except that we must distinguish between the number of
isomorphic copies of an
orbit being zero and being a nonzero multiple of $r{-}1$.

   Conversely, for any finite group $H$,
any finite $H$-set of cardinal congruent to $  m \text{ mod }(r{-}1)$
can be identified with some
$B \in \mathfrak{B}_{m}^{(t)}$ and
hence give a homomorphism from $H$ to $t V_{r,m}$.
\end{background}

\begin{conclusions}\label{concs:hig}
Let us now recall Higman's   recovery of $r$ and $\gcd(m,r{-}1)$
from the isomorphism class of $tV_{r,m}$.

Let $p$ be a prime number, let $a \in [1{\uparrow}\infty[\,$, and
let $\operatorname{cc}(p^a, tV_{r,m})$ denote the
number of conjugacy classes of cyclic subgroups of $tV_{r,m}$ whose
order divides $p^{\,a}$.
Then $\operatorname{cc}(p^{\,a}, tV_{r,m})$ is an invariant of the
isomorphism class of $tV_{r,m}$.
It follows from Background~\ref{back:hig}   that
\begin{align}
&\operatorname{cc}(p^{\,a}, tV_{r,m}) \text{ is equal to the number
of sequences }
    n_{[\mkern-3mu[0{\uparrow}a]\mkern-3mu]}  \in
   [0{\uparrow}(r{-}1)] ^{a+1} \label{eq:hig}\\
&\text{such that}
\textstyle \sum\limits_{j=0}^a(n_jp^{\,j})\equiv m \text{ mod }(r{-}1)
\text{, and }  n_i \ne 0  \text { for some } i \in [0{\uparrow}a]. \nonumber
\end{align}

(i). Let  $p^{\,a}$ be the highest power of $p$ dividing $r{-}1$.
Let \mbox{$p^{\,b}$} denote the highest power of $p$ dividing $
\gcd(m, r{-}1)$.
We shall show that $\operatorname{cc}(p^{\,a}, tV_{r,m})  =
\textstyle \sum\limits_{i=0}^b (p^i r^{a-i})$.

By rewriting \eqref{eq:hig} ignoring  leading zeros in
$n_{[\mkern-3mu[0{\uparrow}a]\mkern-3mu]}$, we see that
$$\operatorname{cc}(p^{\,a}, tV_{r,m})
=\textstyle \sum\limits_{i=0}^a
\vert \{ n_{[\mkern-3mu[i{\uparrow}a]\mkern-3mu]} \in
([0{\uparrow}(
r-1
)])^{a+1-i} :
\sum\limits_{j=i}^a(n_jp^{\,j})\equiv m \text{ mod }(r{-}1)
\text{, and }  n_i \ne 0   \} \vert.$$
If  $b < a$, then   $x p^{\,b+1} \equiv m \text{ mod }(r{-}1)$ has no
solutions, and now, since $b \le a$, we see
$$\operatorname{cc}(p^{\,a}, tV_{r,m}) =
\textstyle \sum\limits_{i=0}^b
\vert \{ n_{[\mkern-3mu[i{\uparrow}a]\mkern-3mu]} \in
([0{\uparrow}(r {-}1)])^{a+1-i} :
\sum\limits_{j=i}^a(n_jp^{\,j})\equiv m \text{ mod }(r{-}1)
\text{ and }  n_i \ne 0   \} \vert.$$
Here, the solutions of $n_ip^{\,i}\equiv m
-\sum\limits_{j=i+1}^a(n_jp^{\,j})\text{ mod }(r{-}1)$, $n_i \ne 0$,
   are given by all possible
$r^{a-i}$ choices for $n_{[\mkern-3mu[(i+1){\uparrow}a]\mkern-3mu]} \in
([0{\uparrow}(r{-}1)])^{a-i}$, and then   $p^{\,i}$ choices for
$n_i$ in the set
$[1{\uparrow}(r{-}1)]$ of representatives of $\integers_{r-1}$.  Hence
$\operatorname{cc}(p^{\,a}, tV_{r,m})  = \textstyle
\sum\limits_{i=0}^b (p^{\,i} r^{a-i})$, as claimed.

(ii). Now suppose that $p$ does not divide  $r{-}1$.

By arguing as in (i), we can show that   $\operatorname{cc}(p^{\,a},
tV_{r,m}) \overset{\eqref{eq:hig}}{=} \sum\limits_{i=0}^a   r^{a-i} $.
The case $a=1$ shows that
for all but finitely many primes $p$, there are exactly $r$ conjugacy
classes of subgroups
of order  exactly $p$ in~$tV_{r,m}$.
It now follows  that $r$ is an invariant  of the isomorphism class of
$tV_{r,m}$.

   It then follows from (i) that   $\gcd(m,r{-}1)$ is also an
invariant  of the isomorphism class
of~$tV_{r,m}$.
\end{conclusions}

\section{The Bleak-Brin-Lanoue proof of invariance of $t$}\label{sec:BBL}

We use the notation of Definitions~\ref{defs:BHT}.

In~\cite{BL}, Bleak-Lanoue developed arguments of
Brin~\cite{Brin},~\cite{Brin0} to
   prove that if \mbox{$t'V_{2,1} \iso tV_{2,1}$} then $t'=t$.
In this section we shall give a straightforward adaptation of
their arguments to our language and show  that if \mbox{$t'V_{r',m'}
\iso tV_{r,m}$}, then $t'=t$.
Here,  $tV_{r,m}$  will be viewed as a group of self-homeomorphisms
of a Cantor set~$\mathcal{E}_{r,m}^{(t)}$;
since the elements of  $\mathcal{E}_{r,m}^{(t)}$ involve one-sided
infinite words,
we follow the standard practice of using \textit{left} actions on
\textit{right}-infinite words.

\begin{definitions} Let $X$ be a topological space and let $G$ be a
group of self-homeomorphisms of $X$
acting on the left, $g:x\mapsto g  {\,\cdot\,} x$.

Let $x \in X$.  We let  $\mathcal{N}(x) $ denote the set of all open
neighbourhoods of $x$ in $X$,
a downward directed system.  We write $\Fix(x;   G ) \coloneq \{g \in
G \mid g {\,\cdot\,} x = x \} \le G$.
For each subset $U$ of $X$, we write $\Fix(U;   G ) \coloneq
\bigcap\limits_{u \in U}
\Fix(u;   G )  \le G$.  We write
   $$\Fix^\circ(x; G) \coloneq \textstyle \bigcup\limits_{U \in
\,\mathcal{N}(x)}
\mkern-10mu\Fix(U;G)\,\, \unlhd \,\,\Fix(x; G)
\text{\,\,\,\,\,\,\,and\,\,\,\,\,}
\Germs(x ;G) \coloneq \Fix(x;G)/\Fix^\circ(x;G),$$
called the \textit{groups of germs of $G$ which fix $x$}.

We say that $G$ is \textit{locally dense} if, for each nonempty, open
subset $U$ of $X$ and each $u \in U$,
the closure of the orbit $\Fix(X{\setminus}U;G) {\,\cdot\,} u$
contains some nonempty, open subset of $U$.
\end{definitions}

To recall   Rubin's theorem,
   we \textit{copy} the following paragraph from~\cite{Brin} and add
to  Rubin's theorem a
phrase from~\cite{BL}  about germs.

\begin{background}[\normalfont\cite{Brin}, \cite{BL}]
\label{back:rub} The following is essentially
Theorem 3.1 of~\cite{Rubin2}
   where it is described as a combination
of parts (a), (b) and (c) of Theorem 3.5 of~\cite{Rubin1}. The
hypothesis that there
be no isolated points was inadvertently omitted from~\cite{Rubin2}
where it is needed. The
terminology \textit{locally dense} is not used in
either~\cite{Rubin2} or ~\cite{Rubin1}. However,
in the absence of isolated points, it implies the notion of
\textit{locally moving} that is
used in~\cite{Rubin2}. The absence of isolated points seems to
correspond to the assumption of ``no atoms"
   in the Boolean algebras
of~\cite{Rubin1}.

\medskip

\noindent\textbf{\ref{back:rub}.1~Rubin's theorem} \cite{Rubin2}.
\textit{Let $G$, resp. $H$, be a locally dense
group of self-homeo\-morphisms of a locally compact, Hausdorff topological
space without isolated points   $X$, resp.~$Y$.   For each isomorphism
   $\phi \colon G \to H$,  there exists a unique homeomorphism
$\tau \colon X \to Y$ with the property that, for each $g\in G$,
$\phi(g)  = \tau g \tau^{-1}$,
and then, for each $x \in X$, $\Germs(x;G) \iso \Germs(\tau(x);H)$.}  \qed
\end{background}

\begin{remarks}\label{rems:tinv}  We shall recall below that
$tV_{r,m}$ can be viewed as a
locally dense group of self-homeo\-morphisms of a Cantor set
$\mathcal{E}_{r,m}^{(t)}$
which is a compact, Hausdorff  topological space  without isolated points.
We shall show that the set of isomorphism classes of groups given by
   $\{ \Germs(\nu ;tV_{r,m}) : \nu \in \mathcal{E}_{r,m}^{(t)}\}$  equals
   the set of isomorphism classes of groups given by
$\{\integers^n : n \in [0{\uparrow}t]\}$.
It will then follow from Rubin's theorem that if
$ t' V_{r',m'} \iso tV_{r,m} $, then $t' = t$.

In~\cite{Brin}, Brin showed that $2V_{2,1} \not \iso V_{r,m}$ by using
Rubin's theorem and a delicate analysis of dynamics and orbit sizes.
In an earlier article~\cite{Brin0}, Brin had considered  germs to study
Thompson's groups~$F$ and $T$.
Bleak-Lanoue~\cite{BL}  combined these two approaches.
They found
the set of isomorphism classes of groups given by
   $\{ \Germs(\nu ;tV_{2,1}) : \nu \in \mathcal{E}_{2,1}^{(t)}\}$
and deduced that if $tV_{2,1} \iso t' V_{2,1}$, then $t' = t$.
Our proof closely follows  theirs.
\end{remarks}

\begin{definitions}\label{defs:geom}

For each $\ell \in [1{\uparrow}t]$, let   $\mathcal{E}_\ell$ denote the set of
right-infinite words in $ y_{\{\ell\} \times  [1{\uparrow}r]}$.
We view $\mathcal{E}_\ell$ as a metric space with
$d(\beta,\gamma)\coloneq (1+|\text{ largest common prefix of
}\beta,\,\gamma\,|\,)^{-1}$.  We view
\mbox{$e_{[1{\uparrow}m]\times\{1\}}$} as a discrete space.
Let
$$\mathcal{E}^{(t)}_{r,m} \coloneq e_{[1{\uparrow}m]\times\{1\}}
\times \mathcal{E}_1
\times \cdots \times \mathcal{E}_t,$$
   and let $\mathcal{E}^{(t)}_{r,m}$ have the product topology.
Then   $  \mathcal{E}^{(t)}_{r,m} $ is a compact, Hausdorff space
without isolated points.

We write each $\nu = (e_{i,1},  \beta_1, \ldots,\beta_t) \in
\mathcal{E}^{(t)}_{r,m}$ as a formal
product  $\nu = e_{i,1}   \beta_1   \cdots   \beta_t$, thought of  as
a limit of  elements of $e_{ i,1 }
\langle \langle y_{[1{\uparrow}t] \times [1{\uparrow}r]} \rangle \rangle$
which have long factors in each $\langle \langle y_{\{\ell\} \times
[1{\uparrow}r]} \rangle \rangle$.
With this formal-product viewpoint, we can define the set of elements of $
e_{[1{\uparrow}m]\times\{1\}} \langle \langle y_{[1{\uparrow}t]
\times [1{\uparrow}r]} \rangle \rangle$
that are \textit{prefixes} of $\nu$.
Let
   $  b \in e_{[1{\uparrow}m]\times\{1\}} \langle \langle
y_{[1{\uparrow}t] \times [1{\uparrow}r]}
\rangle \rangle$.  We define the \textit{shadow} of $b$,
denoted $(b{\blacktriangleleft})$, to be the set of all elements of $
\mathcal{E}^{(t)}_{r,m}$ that have
   $b$ as a prefix.
If $B$ is any basis, then $ \mathcal{E}^{(t)}_{r,m}$ is the disjoint
union of the shadows of the elements of $B$.
   Then $(b{\blacktriangleleft})$ is a closed and open subset of
$ \mathcal{E}^{(t)}_{r,m}$, and the set of all
shadows forms a basis for the open topology on $ \mathcal{E}^{(t)}_{r,m}$.

Let $\integers[\mathcal{E}^{(t)}_{r,m}]$ denote the free abelian group on
$\mathcal{E}^{(t)}_{r,m}$, with the elements of
$\integers[\mathcal{E}^{(t)}_{r,m}]$
   expressed as formal sums $\sum\limits_{\nu \in
\mathcal{E}^{(t)}_{r,m}} n_\nu {\,\cdot\,} \nu$,
with $n_\nu = 0$  for all but finitely many $\nu \in  \mathcal{E}^{(t)}_{r,m}$.
We think of the elements of
$\integers[\mathcal{E}^{(t)}_{r,m}]$ as matrices that can be
approximated arbitrarily closely by
elements of  $\Mat{m}(\L^{\otimes t})\,e_{1,1}$.
In this way,     $\integers[\mathcal{E}^{(t)}_{r,m}]$ has the structure of a
topological left $\Mat{m}(\,\L^{\otimes t})$-module.
We shall see that   $\PU{m}(\,\L^{\otimes t})$  acts on the
$\integers$-basis $\mathcal{E}^{(t)}_{r,m}$.

Let  $\nu \in  \mathcal{E}^{(t)}_{r,m}$ and let
   $a,\, b \in e_{[1{\uparrow}m]\times\{1\}} \langle \langle
y_{[1{\uparrow}t] \times [1{\uparrow}r]}
\rangle \rangle$.

If $b$ is not a prefix of $\nu$,  then $b^\ast {\,\cdot\,} \nu = 0$.

If $\nu \in (b{\blacktriangleleft})$,  we have
$b^\ast {\,\cdot\,} \nu \in \mathcal{E}^{(t)}_{r,m}$,  with  first
factor   $e_{1,1}$,
and  we have $a {\,\cdot\,} b^\ast{\,\cdot\,} \nu  \in
\mathcal{E}^{(t)}_{r,m}$.
   The element
$a {\,\cdot\,} b^\ast \in \Mat{m}(\,\L^{\otimes t})$ uniquely
determines \mbox{$a,\,b \in
e_{[1{\uparrow}m]\times\{1\}}
\langle \langle y_{[1{\uparrow}t] \times [1{\uparrow}r]} \rangle
\rangle$}.  We shall be
viewing $a {\,\cdot\,\,} b^\ast$ as a homeomorphism
$(b{\blacktriangleleft}) \to (a{\blacktriangleleft})$, $\mu \mapsto
a {\,\cdot\,} b^\ast{\,\cdot\,} \mu$,
that  replaces the prefix $b$ with the prefix~$a$.
   This homeomorphism is an identity map if and only if $a=b$. As
homeomorphisms,
$a {\,\cdot\,} b^\ast$ and   $b {\,\cdot\,} a ^\ast$ are mutually inverse.

Let $u \in \PU{m}(\L^{\otimes t})$.  Then there exist bases $A,\, B
\in  \mathfrak{B}_{m}^{(t)}$
and a bijective map $ A \xrightarrow{\phi} B$ such that
\mbox{$u = \sum\limits_{b \in B} (b^{\,\phi^{-1}} \hskip-3.2pt
{\cdot\,} b^\ast)$}.
   Recall that
   the set of all such $(A \xrightarrow{\phi} B)$ forms a single
equivalence class for the
smallest equivalence
relation which identifies expansions.
We view $u$ as being this equivalence class, and we write \mbox{$(A
\xrightarrow{\phi} B)  \in u$}.
Let $\nu \in \mathcal{E}_{r,m}^{(t)}$.
Then there is a unique element $b_0$ of $B$ which is a prefix of $\nu$,
and  $u {\,\cdot\,} \nu  =   b_0^{\,\phi^{-1}} \cdot \,b_0^\ast \cdot
\nu \in \mathcal{E}_{r,m}^{(t)}$.
Left multiplication by $u$ then
gives a self-homeomorphism of $\mathcal{E}_{r,m}^{(t)}$ which acts as
$ b^{\,\phi^{-1}} \hskip-3.2pt {\cdot\,} b^\ast $ on   $
(b{\blacktriangleleft})$, for each $b \in B$.
   The action of $u$ on  $\mathcal{E}_{r,m}^{(t)}$
is trivial only if $u=1$.
Thus  $tV_{r,m}$, identified with $\PU{m}(\L^{\otimes t})$,
   is a  group of self-homeomorphisms of  $\mathcal{E}_{r,m}^{(t)}$.
\end{definitions}

\begin{lemma}[\normalfont Brin~\cite{Brin}]  The group
$tV_{r,m}$  of self-homeomorphisms of  $\mathcal{E}_{r,m}^{(t)}$ is
locally dense.
\end{lemma}

\begin{proof}  Consider an open subset of $\mathcal{E}_{r,m}^{(t)}$
and then choose a smaller open subset
of the form~$(b{\blacktriangleleft})$.  Let
$H \coloneq  \Fix( \mathcal{E}_{r,m}^{(t)} \setminus
(b{\blacktriangleleft}); tV_{r,m})$.
Then   $H$ acts on $(b{\blacktriangleleft})$.  Consider any $\nu \in
(b{\blacktriangleleft})$.
We shall show that the closure of the orbit
$H {\,\cdot\,} \nu$ is all of $(b{\blacktriangleleft})$, which will
show that $tV_{r,m}$ is locally dense.

Choose any $\nu\mkern2mu' \in (b{\blacktriangleleft})$.  We want to
approximate $\nu\mkern2mu'$ arbitrarily closely
by various $u {\,\cdot\,} \nu$ with $u \in H$.  Choose an open
neighbourhood of $\nu\mkern2mu'$  in
$(b{\blacktriangleleft})$, and
then choose a smaller open neighbourhood of the form
$(b'{\blacktriangleleft})$.
It suffices to find $u \in H$ such that  \mbox{$u {\,\cdot\,} \nu \in
(b\mkern2mu'{\blacktriangleleft})$}.

Now $b$ is a prefix of $b\mkern2mu'$ which is a prefix of $\nu\mkern2mu'$.
Let $B$ be a basis containing $b$, and expand $B$ to a basis $B'$ by
expanding $b$
towards $b\mkern2mu'$, that is,  $B\setminus\{b\} \subseteq B'$ and
$b\mkern2mu' \in B'$.  There exists a
unique $a \in B'$ such that $a$ is a prefix of $\nu$, and then $b$ is
a prefix of $a$.
Choose a bijective map
$  B' \xrightarrow{\phi} B'$ that fixes $B\setminus \{b\}$ and sends
$b\mkern2mu'$ to $a$.
Then  $(B' \xrightarrow{\phi} B')$ lies in a unique $u \in tV_{m,r}$.
Now $u \in H$,
   and $u$ carries $(a{\blacktriangleleft})$ to
   $(b\mkern2mu'{\blacktriangleleft})$.  In particular, $u {\,\cdot\,}
\nu \in (b\mkern2mu'{\blacktriangleleft})$,
as desired.
\end{proof}

\begin{conclusions}\label{concs:BBL} Let $\nu = (e_{i,1},
\beta_1,\ldots,\beta_t)
\in \mathcal{E}_{r,m}^{(t)}$.  We  want to analyse $\Germs(\nu, tV_{r,m})$.

Consider any $u \in  \Fix(\nu; tV_{r,m})$, and consider any $(A
\xrightarrow{\phi} B) \in u$.
There exist  a unique $a \in A$ and a unique $b \in B$ such that $a$
and $b$ are prefixes of
$\nu$.  Since
$u \in  \Fix(\nu; tV_{r,m})$, we have $b^{\,\phi^{-1}} \hskip-3.2pt
{\,\cdot\,} b^\ast
   {\,\cdot\,} \nu = \nu$, or, equivalently,
   $ (b^{\,\phi^{-1}})^\ast \hskip-3.2pt {\,\cdot\,} \nu = b^\ast
{\,\cdot\,} \nu$, or, equivalently,
we have two factorizations $\nu = b^{\phi^{-1}} \hskip-3.2pt
{\,\cdot\,} \nu\,' = b{\,\cdot\,} \nu\,'$
with the same \textit{tail} $\nu\,'$.   Thus $a^\phi = b$.  Notice that
$u$ acts as $a {\,\cdot\,} b^\ast$ on   $ (b{\blacktriangleleft})$.  Moreover,
any element of the coset $u {\,\cdot\,}
\Fix((b{\blacktriangleleft});tV_{r,m})$ will also act
as   $a {\,\cdot\,} b^\ast$ on   $ (b{\blacktriangleleft})$.
If in place of $(A \xrightarrow{\phi} B)$ we choose an expansion of
$(A \xrightarrow{\phi} B) $ in $u$, then in place of  $a {\cdot} b^\ast$
   we get an element of the form   $ a {\,\cdot\,}  c {\,\cdot\,}
c^\ast  {\,\cdot\,}  b^\ast$, where
$c$ is a prefix of $\nu\,'$; we then say that
$ a {\,\cdot\,}  c {\,\cdot\,} c^\ast  {\,\cdot\,}  b^\ast$ is an
\textit{expansion of $a {\cdot} b^\ast$
towards $\nu$}.  Here all the elements of the coset
$u {\,\cdot\,} \Fix(((b{\,\cdot\,}c){\blacktriangleleft});tV_{r,m})$   act
as  $ a {\,\cdot\,}  c {\,\cdot\,} c^\ast  {\,\cdot\,}  b^\ast$ on
    $( (b{\,\cdot\,}c){\blacktriangleleft})$.  Thus longer and longer
expansions of $a {\cdot\,\,} b^\ast$
towards $\nu$
determine larger and larger cosets within the germ of $u$.

This leads us to consider the set $\{a {\,\cdot\,} b^\ast \mid a, b \in
e_{ i,1} \langle \langle y_{[1{\uparrow}t] \times [1{\uparrow}r]}
\rangle \rangle,
\,a {\,\cdot\,} b^\ast   {\,\cdot\,} \nu =  \nu\}$ modulo the
smallest equivalence relation that identifies
expansions towards $\nu$.
This set of equivalence classes is a group, denoted $\rep(\nu)$,
with the multiplication that is induced from the multiplication of
compatible representatives,
$(a  {\,\cdot\,} b^\ast) {\,\cdot\,} (b  {\,\cdot\,} c^\ast) = a
{\,\cdot\,}c^\ast$.

It follows from the foregoing that  we have an injective
   homomorphism $$\Germs(\nu; tV_{r,m}) \to \rep(\nu).$$
To see that  this homomorphism is also surjective,  notice that it is
a straightforward matter to construct an element of
$tV_{r,m}$ which carries one given prefix of $\nu$ to another as follows.
We choose one basis containing each,
   and if the bases are not the same size, the smaller basis
can be expanded without removing the specified prefix until the bases
are the same size.  Then we  choose any bijection between the bases
that carries
the first chosen prefix  of $\nu$ to the second chosen prefix of $\nu$.

The next step is to compute $\rep(\nu)$.

Let $\ell \in [1{\uparrow}t]$.
We say that $\beta_\ell$ is \textit{rational} if there exist $w_\ell$,
$z_\ell \in  \langle \langle y_{\{\ell\} \times [1{\uparrow}r]}
\rangle \rangle$ with
$z_\ell \ne 1$ such that $\beta_\ell = w_\ell {\cdot}z_\ell^\infty$;
otherwise, $\beta_\ell$ is \textit{irrational}.

For the purpose of exposition, we may assume that there exists some
$n \in [0{\uparrow}t]$ such that
$\beta_\ell$ is rational if $\ell \in [1{\uparrow}n]$ and
$\beta_\ell$ is irrational if $\ell \in [(n{+}1){\uparrow}t]$.
Thus we may write $$\nu =  (e_{i,1}, w_1 {\cdot} z_1^\infty,  \ldots,
w_n {\cdot} z_n^\infty,  \beta_{n+1} \cdots \beta_t),$$
and we may further assume that each $z_\ell$ is not a proper power.

With $u$, $a$ and $b$ as before,
we can expand the prefixes $a$, $b$ of $\nu$ to longer prefixes and
arrange that
\begin{align*}
   &a^\ast \cdot \nu = b^\ast \cdot \nu = (e_{1,1},  z_1^{\,\infty},
\ldots,  z_n^{\,\infty},  \beta_{n+1}',
\cdots, \beta_t'\,), \text{ and then}
\\&a = e_{i,1}  (w_1 {\cdot} z_1^{q_1}) \cdots   (w_n {\cdot}
z_n^{q_n}) w_{n+1} \cdots w_t,
\\&b = e_{i,1}   (w_1 {\cdot} z_1^{q_1'}) \cdots   (w_n {\cdot}
z_n^{q_n'}) w_{n+1} \cdots w_t,
\end{align*}
where  $q_{[\mkern-3mu[1{\uparrow}n]\mkern-3mu]},\,
q_{[\mkern-3mu[1{\uparrow}n]\mkern-3mu]}'
   \in ([0{\uparrow}\infty[)^n$\vspace{1mm} and
$\beta_\ell = w_\ell {\,\cdot\,} \beta_\ell'$, $\ell \in [(n{+}1){\uparrow}t]$;
notice that irrationality implies  that  tails match up with unique
prefixes, while the fact that $z_\ell$ is not a proper
power implies that the tail $z_\ell^{\,\infty}$ matches up with a
prefix that is unique
up to right multiplication by a power of $z_\ell$.
Then
\begin{equation}\label{eq:word}
a \cdot b^\ast =
e_{i,i} (w_1 {\cdot} z_1^{q_1} {\cdot} z_1^{\ast q_1'} {\cdot} w_1^\ast) \cdots
(w_n {\cdot} z_n^{q_n} {\cdot} z_n^{\ast q_n'} {\cdot} w_n^\ast).
\end{equation}
Thus every element of $\rep(\nu)$ contains an element of the
form~\eqref{eq:word}.
Conversely, every element of the form~\eqref{eq:word} lies in some
element of $\rep(\nu)$.

   It is now straightforward to show that
   $\Germs(\nu, tV_{r,m}) \iso \integers^n$, with elements represented by
the expression~\eqref{eq:word}
corresponding to  $(q_1 \hskip-3pt{\,-\,}q_1',\,\ldots,\,
q_n\hskip-3pt {\,-\,} q_n') \in \integers^n$.

Let us now show that, for each $n \in [0{\uparrow}t]$,  there exists
some $\nu \in \mathcal{E}_{r,m}^{(t)}$ such that
\mbox{$\Germs(\nu, tV_{r,m}) \iso \integers^n$}.  Since there are
only countably many
rational right-infinite words,  there exists some
$\nu = (e_{1,1},\, y_{1,1}^{\,\infty},\, y_{2,1}^{\,\infty},\, \ldots,\,
   y_{n,1}^{\,\infty},\, \beta_{n+1},\, \cdots,\, \beta_n)\in
\mathcal{E}_{r,m}^{(t)},$
such that, for each \mbox{$\ell \in [(n{+}1){\uparrow}t]$},
$\beta_\ell$ is irrational.
By the foregoing, \mbox{$\Germs(\nu, tV_{r,m}) \iso \integers^n$}.

We have now shown that the set of isomorphism classes of groups given
by the set
   $\{ \Germs(\nu ;tV_{r,m}) : \nu \in \mathcal{E}_{r,m}^{(t)}\}$ equals
the set of isomorphism classes of groups given by the set
$\{\integers^n : n \in [0{\uparrow}t]\}$.
It now
follows from Theorem~\ref{back:rub}.1  that if
$ t' V_{r',m'}\iso tV_{r,m}$, then $t' = t$.

In fact, we can say more.  The class of groups isomorphic to
$\integers^n$ is closed
under taking subgroups of finite index.
Any (conjecturally rare) subgroup of finite index in $tV_{r,m}$ is a
locally dense group of
self-homeomorphisms of $\mathcal{E}_{r,m}^{(t)}$ and has the same
$t{+}1$ types of germs as $tV_{r,m}$.
Thus, if $ t' V_{r',m'}$ and $tV_{r,m}$ are commensurable, then $t' = t$.
\end{conclusions}

\section{Summary}\label{sec:summ}

The following builds on work of Abrams, \'Anh, Bleak, Brin, Higman,
Lanoue, Pardo, and  Thompson.

\begin{theorem}\label{thm:main} Let $r_1,\,r_2 \in
[2{\uparrow}\infty[\,$,   $m_1,\, m_2,\,t_1,\,t_2 \in [1{\uparrow}\infty[$\,.
The following are equivalent. \vspace{-2mm}
\begin{enumerate}[\normalfont (a).]
\setlength\itemsep{-.2pt}
\item $r_1 = r_2$, $\gcd(m_1,r_1{-}1)=\gcd(m_2,r_2{-}1)$, and $t_1 = t_2$.
\item $\Mat{m_1}(\L_{r_1}^{\otimes t_1})$ and
$\Mat{m_2}(\L_{r_2}^{\otimes t_2})$ are isomorphic as
partially ordered rings with involution.
\item $t_1V_{r_1,m_1} \iso t_2V_{r_2,m_2}$.
\end{enumerate}
\end{theorem}

\begin{proof} (a) $\Rightarrow$ (b) by Theorem~\ref{final}.

(b) $\Rightarrow$ (c).  Suppose that $\Mat{m_1}(\L_{r_1}^{\otimes t_1}) $ and $
\Mat{m_2}(\L_{r_2}^{\otimes t_2})$ are isomorphic as partially
ordered rings with involution.  Then
   $\PU{m_1} (\,\L_{r_1}^{\otimes t_1}) \iso
\PU{m_2}\mkern1mu(\,\L_{r_2}^{\otimes t_2})$.
Now by Theorem~\ref{thm:Hig},
\mbox{$t_1V_{r_1,m_1} \iso \PU{m_1} (\,\L_{r_1}^{\otimes t_1}) \iso
\PU{m_2}\mkern1mu(\,\L_{r_2}^{\otimes t_2})
\iso t_2V_{r_2,m_2}$}.

(c) $\Rightarrow$ (a). Suppose that $t_1V_{r_1,m_1} \iso t_2V_{r_2,m_2}$.
By Conclusions~\ref{concs:BBL}, $t_1=t_2$.
By Conclusions~\ref{concs:hig}, $r_1 = r_2$ and
$\gcd(m_1,r_1{-}1)=\gcd(m_2,r_2{-}1)$.
\end{proof}

\begin{remarks}\label{rems:leavitt} Let  $r \in [2{\uparrow}\infty[\,$,
$m, \,  t   \in
[1{\uparrow}\infty[\,$, and let $R \coloneq \Mat{m}(\,\L_r^{\otimes
t})$.
\newline\indent\phantom{ii}(i). Pere Ara has shown that
$r$ and $\gcd(m,r{-}1)$ are invariants of the isomorphism class of
$R$ within the class of rings;
we sketch his argument in (ii) below.

Also it follows from work of Jason Bell and George Bergman that
$t$ is an invariant of the isomorphism class of $R$ within the class
of rings; see (iii) below.

Hence the conditions in Theorem~\ref{thm:main} are further equivalent
to  \newline
(b$'$). \textit{$\Mat{m_1}(\L_{r_1}^{\otimes t_1})$ and
$\Mat{m_2}(\L_{r_2}^{\otimes t_2})$ are isomorphic as rings.}
\newline Here, (b) $\Rightarrow$ (b$'$) is clear, while (b$'$)
$\Rightarrow$ (a)
is a consequence of the foregoing results of Ara, Bell and Bergman.
Consequently, with $r$ and $t$ fixed, and $m$ varying,
the set of isomorphism classes of the
rings $\Mat{m}(\,\L_r^{\otimes t})$  is in bijective
correspondence with the set of  positive divisors of  $r{-}1$.
\newline\indent\phantom{i}(ii).  Here we record the argument  of   Ara.

Let $i \in \integers$ and let $A$ be any ring.  We shall use the
homotopy algebraic K-theory groups, $\operatorname{KH}_i(A)$,
introduced by Weibel~\cite{Weibel}.

When we  apply the Ara-Brustenga-Corti\~nas result~\cite[Theorem~8.6]{ABC}  to
  the quiver $E$  with one vertex and $r$ loops, where  $L_A(E)
\coloneq \L_r \otimes_{\integers} A$,
we obtain an exact sequenece
$$\operatorname{KH}_i(A) \xrightarrow{\text{mult. by }r-1}
\operatorname{KH}_i(A) \xrightarrow{\text{natural}}
\operatorname{KH}_i(\L_r \otimes_{\integers} A)
\to \operatorname{KH}_{i-1}(A)\xrightarrow{\text{mult. by }r-1}
\operatorname{KH}_{i-1}(A).$$

If   $A = \integers$, then $\operatorname{KH}_i(A) =0$ if $i < 0$, while
$\operatorname{KH}_0(A) \iso \integers$, with the class of $A$ in
$\operatorname{KH}_0(A)$
corresponding to $1$; see~\cite[Example~1.4]{Weibel}.
It then follows  by induction on~$t$ that if $A =\L_r^{\otimes t}$, then
$\operatorname{KH}_i(A) =0$ if  $i < 0$, while
$\operatorname{KH}_0(A) \iso \integers_{r-1}$
with the class of $A$ in $\operatorname{KH}_0(A)$
corresponding to the class of $1$ in $\integers_{r-1}$.

Recall that $R$ denotes $\Mat{m}(\L_r^{\otimes t})$. It now follows that
$\operatorname{KH}_0(R)
\iso \integers_{r-1}$ with the class of
$R$ corresponding to
the class of $m$ in $\integers_{r-1}$.
Thus $\operatorname{KH}_0(R)$ is cyclic of order $r{-}1$ and
  the class of
$R$ in $\operatorname{KH}_0(R)$ has order  $\frac{r{-}1}{\gcd(m,\, r{-}1)}$.
Hence  $r$ and $\gcd(m,r{-}1)$ are
invariants of the isomorphism class of the ring $R$, as desired.
\newline\indent(iii). Here we build on unpublished work of  Bell
and  Bergman.

Let $K$ be a commutative field, and let $\Gamma \coloneq  K
\otimes_\integers \L_r$.

Let $\Gamma^{\text{op}}$ denote the opposite ring of $\Gamma$.
Let $\Gamma^{\text{e}} \coloneq \Gamma \otimes_K \Gamma^{\text{op}}$.
    Where $K$ is understood,
the projective dimension of the left $\Gamma^{\text{e}}$-module
$\Gamma$ is denoted
$\dim \Gamma$. Bergman-Dicks~\cite[(17) and (4)]{BD} showed that
there exists an exact
sequence of left
$\Gamma^{\text{e}}$-modules $0 \to (\Gamma^{\text{e}})^{r} \to
\Gamma^{\text{e}} \to
\Gamma \to 0$.
Thus $\dim \Gamma \le 1$.

Straightforward  normal-form arguments show that the element $x_1{-}1$ of
$\Gamma$  does not have a left inverse and is not a left zerodivisor; thus
$\operatorname{w.gl.dim} \Gamma \ge 1$.

Since $\dim \Gamma \le 1$ and $\operatorname{w.gl.dim} \Gamma
\ge 1$, the Eilenberg-Rosenberg-Zelinsky result~\cite[Proposition
10(2)]{ERZ} implies that, for each $K$-algebra~$\Lambda$,
$\operatorname{l.gl.dim}(\Lambda\otimes_K \Gamma) = 1+
\operatorname{l.gl.dim}(\Lambda)$, that is,
$\operatorname{l.gl.dim}(\Lambda\otimes_\integers \L_r) = 1+
\operatorname{l.gl.dim}(\Lambda)$.

Now
$\operatorname{l.gl.dim} (K \otimes_\integers R) = \operatorname{l.gl.dim}
  (\Mat{m}(K) \otimes_\integers \L_r^{\otimes t} )  = t$, by induction on $t$.
Thus $t$ is an invariant of the isomorphism class of the ring $R$.
\end{remarks}

\bibliographystyle{plain}

\bigskip

\noindent
\textsc{Warren Dicks}    \\
\textsc{Departament de  Matem\`atiques \\
Universitat Aut\`onoma de Barcelona  \\
08193 Bellaterra (Barcelona), Spain}\\
\noindent \emph{email}{:\;\;}\url{dicks@mat.uab.cat}\\
\noindent \emph{URL}{:\;\;}\url{http://mat.uab.cat/~dicks/}

\bigskip

\noindent
\textsc{Conchita Mart\'{\i}nez-P\'erez}   \\
\textsc{Departamento de Matem\'aticas  \\
Universidad de Zaragoza  \\
50009 Zaragoza, Spain}\\
\noindent \emph{email}{:\;\;}\url{conmar@unizar.es}
\end{document}